\documentclass[12pt,a4paper]{amsart}
\usepackage{amsmath, amsthm, amsfonts, amssymb, txfonts, graphicx, hyperref, bm,verbatim,siunitx}

\theoremstyle{plain}
\newtheorem{thrm}{Theorem}[section]
\newtheorem{lmm}[thrm]{Lemma}
\newtheorem{prpstn}[thrm]{Proposition}

\newtheorem*{cnjctr}{Conjecture}
\newtheorem*{rmk}{Remark}

\numberwithin{sblmm}{thrm} 
\numberwithin{equation}{section}

\renewcommand{\phi}{\varphi}
\begin{document}
\title{Small gaps between primes}
\author{James Maynard}
\address{Centre de recherches math\'ematiques,
Universit\'e de Montr\'eal,
Pavillon Andr\'e-Aisenstadt,
2920 Chemin de la tour, Room 5357,
Montr\'eal (Qu\'ebec) H3T 1J4}
\email{maynardj@dms.umontreal.ca}
\begin{abstract}
We introduce a refinement of the GPY sieve method for studying prime $k$-tuples and small gaps between primes. This refinement avoids previous limitations of the method, and allows us to show that for each $k$, the prime $k$-tuples conjecture holds for a positive proportion of admissible $k$-tuples. In particular, $\liminf_{n}(p_{n+m}-p_n)<\infty$ for every integer $m$. We also show that $\liminf(p_{n+1}-p_n)\le 600$, and, if we assume the Elliott-Halberstam conjecture, that $\liminf_n(p_{n+1}-p_n)\le 12$ and $\liminf_n (p_{n+2}-p_n)\le 600$.
\end{abstract}
\maketitle
\section{Introduction}
We say that a set $\mathcal{H}=\{h_1,\dotsc,h_k\}$ of distinct non-negative integers is `admissible' if, for every prime $p$, there is an integer $a_p$ such that $a_p\nequiv h\pmod{p}$ for all $h\in\mathcal{H}$. We are interested in the following conjecture.

\begin{cnjctr}[Prime $k$-tuples conjecture]
Let $\mathcal{H}=\{h_1,\dots,h_k\}$ be admissible. Then there are infinitely many integers $n$ such that all of $n+h_1$, $\dotsc$, $n+h_k$ are prime.
\end{cnjctr}
When $k>1$ no case of the prime $k$-tuples conjecture is currently known. Work on approximations to the prime $k$-tuples conjecture has been very successful in showing the existence of small gaps between primes, however. In their celebrated paper \cite{GPY}, Goldston, Pintz and Y\i ld\i r\i m introduced a new method for counting tuples of primes, and this allowed them to show that
\begin{equation}
\liminf_n \frac{p_{n+1}-p_n}{\log{p_n}}=0.
\end{equation}
The recent breakthrough of Zhang \cite{Zhang} managed to extend this work to prove
\begin{equation}
\liminf_n(p_{n+1}-p_n)\le \num{70000000},\label{eq:ZhangBound}
\end{equation}
thereby establishing for the first time the existence of infinitely many bounded gaps between primes. Moreover, it follows from Zhang's theorem the that number of admissible sets of size 2 contained in $[1,x]^2$ which satisfy the prime $2$-tuples conjecture is $\gg x^2$ for large $x$. Thus, in this sense, a positive proportion of admissible sets of size 2 satisfy the prime 2-tuples conjecture. The recent polymath project \cite{Polymath} has succeeded in reducing the bound \eqref{eq:ZhangBound} to 4680, by optimizing Zhang's arguments and introducing several new refinements.

The above results have used the `GPY method' to study prime tuples and small gaps between primes, and this method relies heavily on the distribution of primes in arithmetic progressions. Given $\theta>0$, we say the primes have `level of distribution $\theta$'\footnote{We note that different authors have given slightly different names or definitions to this concept. For the purposes of this paper, \eqref{eq:LevelOfDistribution} will be our definition of the primes having level of distribution $\theta$.} if, for every $A>0$, we have
\begin{equation}
\sum_{q\le x^\theta}\max_{(a,q)=1}\Bigl|\pi(x;q,a)-\frac{\pi(x)}{\phi(q)}\Bigr|\ll_A \frac{x}{(\log{x})^A}.\label{eq:LevelOfDistribution}
\end{equation} 
The Bombieri-Vinogradov theorem establishes that the primes have level of distribution $\theta$ for every $\theta<1/2$, and Elliott and Halberstam \cite{ElliottHalberstam} conjectured that this could be extended to every $\theta<1$. Friedlander and Granville \cite{FriedlanderGranville} have shown that \eqref{eq:LevelOfDistribution} cannot hold with $x^{\theta}$ replaced with $x/(\log{x})^B$ for any fixed $B$, and so the Elliott-Halberstam conjecture is essentially the strongest possible result of this type.

The original work of Goldston, Pintz and Y\i ld\i r\i m showed the existence of bounded gaps between primes if \eqref{eq:LevelOfDistribution} holds for some $\theta>1/2$. Moreover, under the Elliott-Halberstam conjecture one had $\liminf_n(p_{n+1}-p_n)\le 16$. The key breakthrough of Zhang's work was in establishing that a slightly weakened form of \eqref{eq:LevelOfDistribution} holds for some $\theta>1/2$.

If one looks for bounded length intervals containing two or more primes, then the GPY method fails to prove such strong results. Unconditionally we are only able to improve upon the trivial bound from the prime number theorem by a constant factor \cite{GPY:ManyPrimes}, and even assuming the Elliott-Halberstam conjecture, the best available result \cite{GPY} is
\begin{equation}
\liminf_n\frac{p_{n+2}-p_n}{\log{p_n}}=0.
\end{equation}
The aim of this paper is to introduce a refinement of the GPY method which removes the barrier of $\theta=1/2$ to establishing bounded gaps between primes, and allows us to show the existence of arbitrarily many primes in bounded length intervals. This answers the second and third questions posed in \cite{GPY} on extensions of the GPY method (the first having been answered by Zhang's result). Our new method also has the benefit that it produces numerically superior results to previous approaches.

\begin{thrm}\label{thrm:ManyPrimes}
Let $m\in\mathbb{N}$. We have
\[\liminf_{n}(p_{n+m}-p_n)\ll m^3 e^{4m}.\]
\end{thrm}
Terence Tao (private communication) has independently proven Theorem \ref{thrm:ManyPrimes} (with a slightly weaker bound) at much the same time. He uses a similar method; the steps are more-or-less the same but the calculations are done differently. We will indicate some of the differences in our proofs as we go along.

We see that the bound in Theorem \ref{thrm:ManyPrimes} is quite far from the conjectural bound of approximately $m\log{m}$ predicted by the prime $m$-tuples conjecture.

Our proof naturally generalizes (but with a weaker upper bound) to many subsequences of the primes which have a level of distribution $\theta>0$. For example, we can show corresponding results where the primes are contained in short intervals $[N,N+N^{7/12+\epsilon}]$ for any $\epsilon>0$ or in an arithmetic progression modulo $q\ll (\log{N})^A$. In particular, our method gives results for simultaneously prime values of linear functions, which might have specific interest. Given $k$ distinct linear functions $L_i(n)=a_in+b_i$ ($1\le i\le k$) with positive integer coefficients such that the product function $\Pi(n)=\prod_{i=1}^kL_i(n)$ has no fixed prime divisor, the method presented here shows that there are infinitely many integers $n$ such that at least $(1/4+o_{k\rightarrow\infty}(1))\log{k}$ of the $L_i(n)$ are prime.

\begin{thrm}\label{thrm:kTuples}
Let $m\in\mathbb{N}$. Let $r\in\mathbb{N}$ be sufficiently large depending on $m$, and let $\mathcal{A}=\{a_1,a_2,\dotsc,a_r\}$ be a set of $r$ distinct integers. Then we have
\[\frac{\#\{\{h_1,\dotsc,h_m\}\subseteq \mathcal{A}:\text{for infinitely many $n$ all of $n+h_1$, $\dotsc$, $n+h_m$ are prime}\}}{\#\{\{h_1,\dotsc,h_m\}\subseteq \mathcal{A}\}}\gg_m 1.\]
\end{thrm}
Thus a positive proportion of admissible $m$-tuples satisfy the prime $m$-tuples conjecture for every $m$, in an appropriate sense.
\begin{thrm}\label{thrm:Unconditional}We have
\[\liminf_n(p_{n+1}-p_n)\le 600.\]
\end{thrm}
We emphasize that the above result does not incorporate any of the technology used by Zhang to establish the existence of bounded gaps between primes. The proof is essentially elementary, relying only on the Bombieri-Vinogradov theorem. Naturally, if we assume that the primes have a higher level of distribution, then we can obtain stronger results.
\begin{thrm}\label{thrm:EH}
Assume that the primes have level of distribution $\theta$ for every $\theta<1$. Then
\begin{align*}\liminf_n(p_{n+1}-p_n)&\le 12,\\
\liminf_n(p_{n+2}-p_n)&\le 600.\end{align*}
\end{thrm}
Although the constant \num{12} of Theorem \ref{thrm:EH} appears to be optimal with our method in its current form, the constant \num{600} appearing in Theorem \ref{thrm:Unconditional} and Theorem \ref{thrm:EH} is certainly not optimal. By performing further numerical calculations our method could produce a better bound, and also most of the ideas of Zhang's work (and the refinements produced by the polymath project) should be able to be combined with this method to reduce the constant further. We comment that the assumption of the Elliott-Halberstam conjecture allows us to improve the bound on Theorem \ref{thrm:ManyPrimes} to $O(m^3e^{2m})$.
\section{An improved GPY sieve method}
We first give an explanation of the key idea behind our new approach. The basic idea of the GPY method is, for a fixed admissible set $\mathcal{H}=\{h_1,\dotsc,h_k\}$, to consider the sum
\begin{align}
S(N,\rho)&=\sum_{N\le n<2N}\Bigl(\sum_{i=1}^k\chi_{\mathbb{P}}(n+h_i)-\rho\Bigr)w_n.\label{eq:BasicSum}
\end{align}
Here $\chi_{\mathbb{P}}$ is the characteristic function of the primes, $\rho>0$ and $w_n$ are non-negative weights. If we can show that $S(N,\rho)>0$ then at least one term in the sum over $n$ must have a positive contribution. By the non-negativity of $w_n$, this means that there must be some integer $n\in[N,2N]$ such that at least $\lfloor\rho+1\rfloor$ of the $n+h_i$ are prime. (Here $\lfloor x\rfloor$ denotes the largest integer less than or equal to $x$.) Thus if $S(N,\rho)>0$ for all large $N$, there are infinitely many integers $n$ for which at least $\lfloor\rho+1\rfloor$ of the $n+h_i$ are prime (and so there are infinitely many bounded length intervals containing $\lfloor\rho+1\rfloor$ primes).

The weights $w_n$ are typically chosen to mimic Selberg sieve weights. Estimating \eqref{eq:BasicSum} can be interpreted as a `$k$-dimensional' sieve problem. The standard Selberg $k$-dimensional weights (which can be shown to be essentially optimal in other contexts) are
\begin{equation}
w_n=\Bigl(\sum_{\substack{d|\prod_{i=1}^k(n+h_i)\\ d<R}}\lambda_d\Bigr)^2,\qquad \lambda_d=\mu(d)(\log{R/d})^k.
\end{equation}
With this choice we find that we just fail to prove the existence of bounded gaps between primes if we assume the Elliott-Halberstam conjecture. The key new idea in the paper of Goldston, Pintz and Y\i ld\i r\i m \cite{GPY} was to consider more general sieve weights of the form
\begin{equation}
\lambda_d=\mu(d)F(\log{R/d}),\label{eq:GPYWeights}
\end{equation}
for a suitable smooth function $F$. Goldston, Pintz and Y\i ld\i r\i m chose $F(x)=x^{k+l}$ for suitable $l\in\mathbb{N}$, which has been shown to be essentially optimal when $k$ is large. This allows us to gain a factor of approximately \num{2} for large $k$ over the previous choice of sieve weights. As a result we just fail to prove bounded gaps using the fact that the primes have exponent of distribution $\theta$ for any $\theta<1/2$, but succeed in doing so if we assume they have level of distribution $\theta>1/2$. 

The new ingredient in our method is to consider a more general form of the sieve weights
\begin{equation}
w_n=\Bigl(\sum_{d_i|n+h_i\forall i}\lambda_{d_1,\dotsc,d_k}\Bigr)^2.\label{eq:WChoice}
\end{equation}
Using such weights with $\lambda_{d_1,\dotsc,d_k}$ is the key feature of our method. It allows us to improve on the previous choice of sieve weights by an arbitrarily large factor, provided that $k$ is sufficiently large. It is the extra flexibility gained by allowing the weights to depend on the divisors of each factor individually which gives this improvement.

The idea to use such weights is not entirely new. Selberg \cite[Page 245]{Selberg:CollectedWorks} suggested the possible use of similar weights in his work on approximations to the twin prime problem, and Goldston and Y\i ld\i r\i m \cite{GoldstonYildirim} considered similar weights in earlier work on the GPY method, but with the support restricted to $d_i<R^{1/k}$ for all $i$.

We comment that our choice of $\lambda_{d_1,\dots,d_k}$ will look like
\begin{equation}
\lambda_{d_1,\dots,d_k}\approx \Bigl(\prod_{i=1}^k\mu(d_i)\Bigr)f(d_1,\dots,d_k),\label{eq:ApproximateLambda}
\end{equation}
for a suitable smooth function $f$. For our precise choice of $\lambda_{d_1,\dots,d_k}$ (given in Proposition \ref{prpstn:MainProp}) we find it convenient to give a slightly different form of $\lambda_{d_1,\dots,d_k}$, but weights of the form \eqref{eq:ApproximateLambda} should produce essentially the same results.
\section{Notation}
We shall view $k$ as a fixed integer, and $\mathcal{H}=\{h_1,\dots,h_k\}$ as a fixed admissible set. In particular, any constants implied by the asymptotic notation $o$, $O$ or $\ll$ may depend on $k$ and $\mathcal{H}$. We will let $N$ denote a large integer, and all asymptotic notation should be interpreted as referring to the limit $N\rightarrow\infty$.

All sums, products and suprema will be assumed to be taken over variables lying in the natural numbers $\mathbb{N}=\{1,2,\dots\}$ unless specified otherwise. The exception to this is when sums or products are over a variable $p$, which instead will be assumed to lie in the prime numbers $\mathbb{P}=\{2,3,\dots,\}$.

Throughout the paper, $\phi$ will denote the Euler totient function, $\tau_r(n)$ the number of ways of writing $n$ as a product of $r$ natural numbers and $\mu$ the Moebius function. We will let $\epsilon$ be a fixed positive real number, and we may assume without further comment that $\epsilon$ is sufficiently small at various stages of our argument. We let $p_n$ denote the $n^{th}$ prime, and $\#\mathcal{A}$ denote the number of elements of a finite set $\mathcal{A}$. We use $\lfloor x\rfloor$ to denote the largest integer $n\le x$, and $\lceil x\rceil$ the smallest integer $n\ge x$. We let $(a,b)$ be the greatest common divisor of integers $a$ and $b$. Finally, $[a,b]$ will denote the closed interval on the real line with endpoints $a$ and $b$, except for in Section 5 where it will denote the least common multiple of integers $a$ and $b$ instead.
\section{Outline of the proof}\label{sctn:Outline}
We will find it convenient to choose our weights $w_n$ to be zero unless $n$ lies in a fixed residue class $v_0\pmod{W}$, where $W=\prod_{p\le D_0}p$. This is a technical modification which removes some minor complications in dealing with the effect of small prime factors. The precise choice of $D_0$ is unimportant, but it will suffice to choose
\begin{equation}D_0=\log\log\log{N},\end{equation}
so certainly $W\ll (\log\log{N})^2$ by the prime number theorem. By the Chinese remainder theorem, we can choose $v_0$ such that $v_0+h_i$ is coprime to $W$ for each $i$ since $\mathcal{H}$ is admissible. When $n\equiv v_0\pmod{W}$, we choose our weights $w_n$ of the form \eqref{eq:WChoice}. We now wish to estimate the sums
\begin{align}
S_1&=\sum_{\substack{N\le n<2N\\n \equiv v_0\pmod{W}}}\left(\sum_{d_i|n+h_i\forall i}\lambda_{d_1,\dotsc,d_k}\right)^2,\\
S_2&=\sum_{\substack{N\le n<2N\\n \equiv v_0\pmod{W}}}\Bigl(\sum_{i=1}^k\chi_{\mathbb{P}}(n+h_i)\Bigr)\left(\sum_{d_i|n+h_i\forall i}\lambda_{d_1,\dotsc,d_k}\right)^2.
\end{align}
We evaluate these sums using the following proposition.
\begin{prpstn}\label{prpstn:MainProp}
 Let the primes have exponent of distribution $\theta>0$, and let $R=N^{\theta/2-\delta}$ for some small fixed $\delta>0$. Let $\lambda_{d_1,\dotsc,d_k}$ be defined in terms of a fixed smooth function $F$ by
\[\lambda_{d_1,\dotsc,d_k}=\Bigl(\prod_{i=1}^k\mu(d_i)d_i\Bigr)\sum_{\substack{r_1,\dotsc,r_k\\d_i|r_i\forall i\\(r_i,W)=1\forall i}}\frac{\mu(\prod_{i=1}^kr_i)^2}{\prod_{i=1}^k\phi(r_i)}F\left(\frac{\log{r_1}}{\log{R}},\dotsc,\frac{\log{r_k}}{\log{R}}\right),\]
whenever $(\prod_{i=1}^kd_i,W)=1$, and let $\lambda_{d_1,\dotsc,d_k}=0$ otherwise. Moreover, let $F$ be supported on $\mathcal{R}_k=\{(x_1,\dotsc,x_k)\in[0,1]^k:\sum_{i=1}^kx_i\le 1\}$.
Then we have
\begin{align*}
S_1&=\frac{(1+o(1)) \phi(W)^k N (\log{R})^k}{W^{k+1}} I_k(F),\\
S_2&=\frac{(1+o(1)) \phi(W)^k N (\log{R})^{k+1}}{W^{k+1}\log{N}}\sum_{m=1}^kJ_k^{(m)}(F),
\end{align*}
provided $I_k(F)\ne 0$ and $J_k^{(m)}(F)\ne 0$ for each $m$, where
\begin{align*}
I_k(F)&=\int_0^1\dotsi \int_0^1 F(t_1,\dotsc, t_k)^2dt_1\dotsc dt_k,\\
J_k^{(m)}(F)&=\int_0^1\dotsi \int_0^1 \left(\int_0^1 F(t_1,\dotsc,t_k)dt_m\right)^2 dt_1\dotsc dt_{m-1} dt_{m+1}\dotsc dt_k.
\end{align*}
\end{prpstn}
We recall that if $S_2$ is large compared to $S_1$, then using the GPY method we can show that there are infinitely many integers $n$ such that several of the $n+h_i$ are prime. The following proposition makes this precise.
\begin{prpstn}\label{prpstn:Explicit}
Let the primes have level of distribution $\theta>0$. Let $\delta>0$ and $\mathcal{H}=\{h_1,\dotsc,h_k\}$ be an admissible set. Let $I_k(F)$ and $J_k^{(m)}(F)$ be given as in Proposition \ref{prpstn:MainProp}, and let $\mathcal{S}_k$ denote the set of Riemann-integrable functions $F:[0,1]^k\rightarrow\mathbb{R}$ supported on $\mathcal{R}_k=\{(x_1,\dotsc,x_k)\in[0,1]^k:\sum_{i=1}^kx_i\le 1\}$ with $I_k(F)\ne 0$ and $J_k^{(m)}(F)\ne 0$ for each $m$. Let
\[M_k=\sup_{F\in\mathcal{S}_k}\frac{\sum_{m=1}^kJ_k^{(m)}(F)}{I_k(F)},\qquad\qquad r_k=\Bigl\lceil\frac{\theta M_k}{2}\Bigr\rceil.\]
Then there are infinitely many integers $n$ such that at least $r_k$ of the $n+h_i$ ($1\le i\le k$) are prime. In particular, $\liminf_n(p_{n+r_k-1}-p_n)\le \max_{1\le i,j\le k}(h_i-h_j)$.
\end{prpstn}
\begin{proof}[Proof of Proposition \ref{prpstn:Explicit}]
We let $S=S_2-\rho S_1$, and recall that from Section 2 that if we can show $S>0$ for all large $N$, then there are infinitely many integers $n$ such that at least $\lfloor \rho+1\rfloor$ of the $n+h_i$ are prime. 

We put $R=N^{\theta/2-\delta}$ for a small $\delta>0$. By the definition of $M_k$, we can choose $F_0\in\mathcal{S}_k$ such that $\sum_{m=1}^kJ_k^{(m)}(F_0)>(M_k-\delta)I_k(F_0)>0$. Since $F_0$ is Riemann-integrable, there is a smooth function $F_1$ such that $\sum_{m=1}^kJ_k^{(m)}(F_1)>(M_k-2\delta)I_k(F_1)>0$. Using Proposition \ref{prpstn:MainProp}, we can then choose $\lambda_{d_1,\dotsc,d_k}$ such that
\begin{align}
S&=\frac{\phi(W)^kN(\log{R})^k}{W^{k+1}}\Bigl(\frac{\log{R}}{\log{N}}\sum_{j=1}^kJ_k^{(m)}(F_1)-\rho I_k(F_1)+o(1)\Bigr)\nonumber\\
&\ge \frac{\phi(W)^kN(\log{R})^k I_k(F_1)}{W^{k+1}}\Bigl(\Bigl(\frac{\theta}{2}-\delta\Bigr)\Bigl(M_k-2\delta\Bigr)-\rho+o(1)\Bigr).
\end{align}
If $\rho=\theta M_k/2-\epsilon$ then, by choosing $\delta$ suitably small (depending on $\epsilon$), we see that $S>0$ for all large $N$. Thus there are infinitely many integers $n$ for which at least $\lfloor \rho+1\rfloor$ of the $n+h_i$ are prime. Since $\lfloor \rho+1\rfloor= \lceil \theta M_k/2\rceil$ if $\epsilon$ is suitably small, we obtain Proposition \ref{prpstn:Explicit}.
\end{proof}
Thus, if the primes have a fixed level of distribution $\theta$, to show the existence of many of the $n+h_i$ being prime for infinitely many $n\in\mathbb{N}$ we only require a suitable lower bound for $M_k$. The following proposition establishes such a bound for different values of $k$.
\begin{prpstn}\label{prpstn:FChoices}Let $k\in\mathbb{N}$, and $M_k$ be as given by Proposition \ref{prpstn:Explicit}. Then
\begin{enumerate}
\item We have $M_5>2$.
\item We have $M_{105}>4$.
\item If $k$ is sufficiently large, we have $M_k>\log{k}-2\log\log{k}-2$.
\end{enumerate}
\end{prpstn}
We now prove Theorems \ref{thrm:ManyPrimes}, \ref{thrm:kTuples}, \ref{thrm:Unconditional} and \ref{thrm:EH} from Propositions \ref{prpstn:Explicit}  and \ref{prpstn:FChoices}. 

First we consider Theorem \ref{thrm:Unconditional}. We take $k=105$. By Proposition \ref{prpstn:FChoices}, we have $M_{105}>4$. By the Bombieri-Vinogradov theorem, the primes have level of distribution $\theta=1/2-\epsilon$ for every $\epsilon>0$. Thus, if we take $\epsilon$ sufficiently small, we have $\theta M_{105}/2>1$. Therefore, by Proposition \ref{prpstn:Explicit}, we have $\liminf(p_{n+1}-p_n)\le \max_{1\le i,j\le 105}(h_i-h_j)$ for any admissible set $\mathcal{H}=\{h_1,\dotsc,h_{105}\}$. By computations performed by Thomas Engelsma (unpublished), we can choose\footnote{Explicitly, we can take $\mathcal{H}=\{$0, 10, 12, 24, 28,  30, 34, 42, 48, 52, 54,  64, 70, 72, 78, 82, 90, 94,  100, 112, 114, 118, 120,  124, 132, 138, 148, 154,  168, 174, 178, 180, 184,  190, 192, 202, 204, 208,  220, 222, 232, 234, 250,  252, 258, 262, 264, 268,  280, 288, 294, 300, 310,  322, 324, 328, 330, 334,  342, 352, 358, 360, 364,  372, 378, 384, 390, 394,  400, 402, 408, 412, 418,  420, 430, 432, 442, 444,  450, 454, 462, 468, 472,  478, 484, 490, 492, 498,  504, 510, 528, 532, 534,  538, 544, 558, 562, 570,  574, 580, 582, 588, 594,  598, 600$\}.$ This set was obtained from the website \url{http://math.mit.edu/~primegaps/} maintained by Andrew Sutherland.} $\mathcal{H}$ such that $0\le h_1<\dotsc<h_{105}$ and $h_{105}-h_1=600$. This gives Theorem \ref{thrm:Unconditional}.

If we assume the Elliott-Halberstam conjecture then the primes have level of distribution $\theta=1-\epsilon$. First we take $k=105$, and see that $\theta M_{105}/2>2$ for $\epsilon$ sufficiently small (since $M_{105}>4$). Therefore, by Proposition \ref{prpstn:Explicit}, $\liminf_n(p_{n+2}-p_n)\le \max_{1\le i,j\le 105}(h_i-h_j)$. Thus, choosing the same admissible set $\mathcal{H}$ as above, we see $\liminf_n(p_{n+2}-p_n)\le 600$ under the Elliott-Halberstam conjecture. 

Next we take $k=5$ and $\mathcal{H}=\{0,2,6,8,12\}$, with $\theta=1-\epsilon$ again. By Proposition \ref{prpstn:FChoices} we have $M_5>2$, and so $\theta M_5/2>1$ for $\epsilon$ sufficiently small. Thus, by Proposition \ref{prpstn:Explicit}, $\liminf_n(p_{n+1}-p_n)\le 12$ under the Elliott-Halberstam conjecture. This completes the proof of Theorem \ref{thrm:EH}.

Finally, we consider the case when $k$ is large. For the rest of this section, any constants implied by asymptotic notation will be independent of $k$. By the Bombieri-Vinogradov theorem, we can take $\theta=1/2-\epsilon$. Thus, by Proposition \ref{prpstn:FChoices}, we have for $k$ sufficiently large
\begin{equation}
\frac{\theta M_k}{2}\ge \Bigl(\frac{1}{4}-\frac{\epsilon}{2}\Bigr)\Bigl(\log{k}-2\log\log{k}-2\Bigr).
\end{equation}
We choose $\epsilon=1/k$, and see that $\theta M_k/2> m$ if $k\ge C m^2 e^{4m}$ for some absolute constant $C$ (independent of $m$ and $k$). Thus, for any admissible set $\mathcal{H}=\{h_1,\dotsc,h_k\}$ with $k\ge C m^2 e^{4m}$, at least $m+1$ of the $n+h_i$ must be prime for infinitely many integers $n$. We can choose our set $\mathcal{H}$ to be the set $\{p_{\pi(k)+1},\dotsc,p_{\pi(k)+k}\}$ of the first $k$ primes which are greater than $k$. This is admissible, since no element is a multiple of a prime less than $k$ (and there are $k$ elements, so it cannot cover all residue classes modulo any prime greater than $k$.) This set has diameter $p_{\pi(k)+k}-p_{\pi(k)+1}\ll k\log{k}$. Thus $\liminf_{n}(p_{n+m}-p_n)\ll k\log{k}\ll m^3 e^{4m}$ if we take $k=\lceil Cm^2 e^{4m}\rceil $. This gives Theorem \ref{thrm:ManyPrimes}.

We can now establish Theorem \ref{thrm:kTuples} by a simple counting argument. Given $m$, we let $k=\lceil Cm^2e^{4m}\rceil$ as above. Therefore if $\{h_1,\dotsc,h_k\}$ is admissible, then there exists a subset $\{h_1',\dotsc,h_m'\}\subseteq\{h_1,\dotsc,h_k\}$ with the property that there are infinitely many integers $n$ for which all of the $n+h'_i$ are prime ($1\le i\le m$).

We let $\mathcal{A}_2$ denote the set formed by starting with the given set $\mathcal{A}=\{a_1,\dotsc,a_r\}$, and  for each prime $p\le k$ in turn removing all elements of the residue class modulo $p$ which contains the fewest integers. We see that $\#\mathcal{A}_2\ge r\prod_{p\le k}(1-1/p)\gg_m r$. Moreover, any subset of $\mathcal{A}_2$ of size $k$ must be admissible, since it cannot cover all residue classes modulo $p$ for any prime $p\le k$. We let $s=\#\mathcal{A}_2$, and since $r$ is taken sufficiently large in terms of $m$, we may assume that $s>k$.

We see there are $\binom{s}{k}$ sets $\mathcal{H}\subseteq\mathcal{A}_2$ of size $k$. Each of these is admissible, and so contains at least one subset $\{h_1',\dotsc,h_m'\}\subseteq\mathcal{A}_2$ which satisfies the prime $m$-tuples conjecture. Any admissible set $\mathcal{B}\subseteq\mathcal{A}_2$ of size $m$ is contained in $\binom{s-m}{k-m}$ sets $\mathcal{H}\subseteq\mathcal{A}_2$ of size $k$. Thus there are at least $\binom{s}{k}\binom{s-m}{k-m}^{-1}\gg_m s^m\gg_m r^m$ admissible sets $\mathcal{B}\subseteq\mathcal{A}_2$ of size $m$ which satisfy the prime $m$-tuples conjecture. Since there are $\binom{r}{m}\le r^m$ sets $\{h_1,\dotsc,h_m\}\subseteq\mathcal{A}$, Theorem \ref{thrm:kTuples} holds.

We are left to establish Propositions \ref{prpstn:MainProp} and \ref{prpstn:FChoices}.
\section{Selberg sieve manipulations}
In this section we perform initial manipulations towards establishing Proposition \ref{prpstn:MainProp}. These arguments are multidimensional generalizations of the sieve arguments of \cite{GGPY}. In particular, our approach is based on the elementary combinatorial ideas of Selberg. The aim is to introduce a change of variables to rewrite our sums $S_1$ and $S_2$ in a simpler form.

Throughout the rest of the paper we assume that the primes have a fixed level of distribution $\theta$, and $R=N^{\theta/2-\delta}$. We restrict the support of $\lambda_{d_1,\dotsc,d_k}$ to tuples for which the product $d=\prod_{i=1}^k d_i$ is less than $R$ and also satisfies $(d,W)=1$ and $\mu(d)^2=1$. We note that the condition $\mu(d)^2=1$ implies that $(d_i,d_j)=1$ for all $i\ne j$.
\begin{lmm}\label{lmm:S1Expression1}
Let
\[y_{r_1,\dotsc,r_k}=\Bigl(\prod_{i=1}^k\mu(r_i)\phi(r_i)\Bigr)\sum_{\substack{d_1,\dotsc,d_k\\ r_i|d_i\forall i}}\frac{\lambda_{d_1,\dotsc,d_k}}{\prod_{i=1}^kd_i}.\]
Let $y_{max}=\sup_{r_1,\dotsc,r_k}|y_{r_1,\dotsc,r_k}|$. Then
\[S_1=\frac{N}{W}\sum_{r_1,\dotsc,r_k}\frac{y_{r_1,\dotsc,r_k}^2}{\prod_{i=1}^k\phi(r_i)}+O\Bigl(\frac{y_{max}^2 \phi(W)^k N(\log{R})^k}{W^{k+1} D_0}\Bigr).\]
\end{lmm}
\begin{proof}
We expand out the square, and swap the order of summation to give\begin{equation}S_1=\sum_{\substack{N\le n<2N\\ n\equiv v_0\pmod{W}}}\Bigl(\sum_{d_i|n+h_i\forall i}\lambda_{d_1,\dotsc,d_k}\Bigr)^2=\sum_{\substack{d_1,\dotsc, d_k\\e_1,\dotsc, e_k}}\lambda_{d_1,\dotsc,d_k}\lambda_{e_1,\dotsc,e_k}\sum_{\substack{N\le n<2N\\ n\equiv v_0\pmod{W}\\ [d_i,e_i]|n+h_i \forall i}}1.
\end{equation}
We recall that here, and throughout this section,  we are using $[a,b]$ to denote the least common multiple of $a$ and $b$.

By the Chinese remainder theorem, the inner sum can be written as a sum over a single residue class modulo $q=W\prod_{i=1}^k[d_i,e_i]$, provided that the integers $W,[d_1,e_1],\dotsc, [d_k,e_k]$ are pairwise coprime. In this case the inner sum is $N/q+O(1)$. If the integers are not pairwise coprime then the inner sum is empty. This gives
\begin{align}
S_1=\frac{N}{W}\sideset{}{'}\sum_{\substack{d_1,\dotsc,d_k\\e_1,\dotsc,e_k}}\frac{\lambda_{d_1,\dotsc,d_k}\lambda_{e_1,\dotsc,e_k}}{\prod_{i=1}^k[d_i,e_i]}+O\Bigl(\sideset{}{'}\sum_{\substack{d_1,\dotsc,d_k\\e_1,\dotsc,e_k}}|\lambda_{d_1,\dotsc,d_k}\lambda_{e_1,\dotsc,e_k}|\Bigr),
\end{align}
where $\sum'$ is used to denote the restriction that we require $W, [d_1,e_1], \dotsc ,  [d_k,e_k]$ to be pairwise coprime. To ease notation we will put $\lambda_{max}=\sup_{d_1,\dotsc, d_k}|\lambda_{d_1,\dotsc,d_k}|$. We now see that since $\lambda_{d_1,\dotsc,d_k}$ is non-zero only when $\prod_{i=1}^kd_i<R$, the error term contributes
\begin{equation}\ll \lambda_{max}^2\Bigl(\sum_{d<R}\tau_k(d)\Bigr)^2\ll \lambda_{max}^2R^2(\log{R})^{2k},\end{equation}
which will be negligible.

In the main sum we wish to remove the dependencies between the $d_i$ and the $e_j$ variables. We use the identity
\begin{equation}\frac{1}{[d_i,e_i]}=\frac{1}{d_ie_i}\sum_{u_i|d_i,e_i}\phi(u_i)\end{equation}
to rewrite the main term as
\begin{equation}
\frac{N}{W}\sum_{u_1,\dotsc,u_k}\Bigl(\prod_{i=1}^k \phi(u_i)\Bigr)\sideset{}{'}\sum_{\substack{d_1,\dotsc,d_k\\e_1,\dotsc,e_k\\u_i|d_i,e_i\forall i}}\frac{\lambda_{d_1,\dotsc,d_k}\lambda_{e_1,\dotsc,e_k}}{(\prod_{i=1}^kd_i) (\prod_{i=1}^ke_i)}.
\end{equation}
We recall that $\lambda_{d_1,\dotsc,d_k}$ is supported on integers $d_1,\dotsc,d_k$ with $(d_i,W)=1$ for each $i$ and $(d_i,d_j)=1$ for all $i\ne j$. Thus we may drop the requirement that $W$ is coprime to each of the $[d_i,e_i]$ from the summation, since these terms have no contribution. Similarly, we may drop the requirement that the $d_i$ variables are all pairwise coprime, and the requirement that the $e_i$ variables are all pairwise coprime. Thus the only remaining restriction coming from the pairwise coprimality of $W,[d_1,e_1],\dotsc,[d_k,e_k]$ is that $(d_i,e_j)=1$ for all $i\ne j$.

We can remove the requirement that $(d_i,e_j)=1$ by multiplying our expression by $\sum_{s_{i,j}|d_i,e_j}\mu(s_{i,j})$. We do this for all $i, j$ with $i\ne j$. This transforms the main term to
\begin{equation}
\frac{N}{W}\sum_{u_1,\dotsc,u_k}\Bigl(\prod_{i=1}^k\phi(u_i)\Bigr)\sum_{s_{1,2},\dotsc,s_{k,k-1}}\Bigl(\prod_{\substack{1\le i,j\le k\\i\ne j}}\mu(s_{i,j})\Bigr)\sum_{\substack{d_1,\dotsc, d_k\\ e_1,\dotsc,e_k\\ u_i|d_i,e_i\forall i\\ s_{i,j}|d_i,e_j\forall i\ne j}}\frac{\lambda_{d_1,\dotsc,d_k}\lambda_{e_1,\dotsc,e_k}}{(\prod_{i=1}^kd_i)(\prod_{i=1}^k e_i)}.\label{eq:S1MainTerm}
\end{equation}
We can restrict the $s_{i,j}$ to be coprime to $u_i$ and $u_j$, because terms with $s_{i,j}$ not coprime to $u_i$ or $u_j$ make no contribution to our sum. This is because $\lambda_{d_1,\dots,d_k}=0$ unless $(d_i,d_j)=1$. Similarly we can further restrict our sum so that $s_{i,j}$ is coprime to $s_{i,a}$ and $s_{b,j}$ for all $a\ne j$ and $b\ne i$. We denote the summation over $s_{1,2},\dotsc,s_{k,k-1}$ with these restrictions by $\sum^*$.

We now introduce a change of variables to make the estimation of the sum more straightforward. We let
\begin{equation}
y_{r_1,\dotsc,r_k}=\Bigl(\prod_{i=1}^k\mu(r_i)\phi(r_i)\Bigr)\sum_{\substack{d_1,\dotsc,d_k\\ r_i|d_i\forall i}}\frac{\lambda_{d_1,\dotsc,d_k}}{\prod_{i=1}^kd_i}.\label{eq:YLambdaDef}
\end{equation}
This change is invertible. For $d_1,\dotsc,d_k$ with $\prod_{i=1}^kd_i$ square-free we find that
\begin{align}
\sum_{\substack{r_1,\dotsc, r_k\\ d_i|r_i\forall i}}\frac{y_{r_1,\dotsc,r_k}}{\prod_{i=1}^k\phi(r_i)}&=\sum_{\substack{r_1,\dotsc,r_k\\ d_i|r_i\forall i}}\Bigl(\prod_{i=1}^k\mu(r_i)\Bigr)\sum_{\substack{e_1,\dotsc,e_k\\r_i|e_i\forall i}}\frac{\lambda_{e_1,\dotsc,e_k}}{\prod_{i=1}^ke_i}\nonumber\\
&=\sum_{e_1,\dotsc,e_k}\frac{\lambda_{e_1,\dotsc,e_k}}{\prod_{i=1}^ke_i}\sum_{\substack{r_1,\dotsc,r_k\\d_i|r_i\forall i\\ r_i|e_i\forall i}}\prod_{i=1}^k\mu(r_i)=\frac{\lambda_{d_1,\dotsc,d_k}}{\prod_{i=1}^k\mu_i(d_i)d_i}.\label{eq:LambdaYDef}
\end{align}
Thus any choice of $y_{r_1,\dotsc,r_k}$ supported on $r_1,\dotsc,r_k$, with the product $r=\prod_{i=1}
^kr_i$ square-free and satisfying $r<R$ and $(r,W)=1$, will give a suitable choice of $\lambda_{d_1,\dotsc,d_k}$. We let $y_{max}=\sup_{r_1,\dotsc, r_k}|y_{r_1,\dotsc,r_k}|$. Now, since $d/\phi(d)=\sum_{e|d}1/\phi(e)$ for square-free $d$, we find by taking $r'=\prod_{i=1}^kr_i/d_i$ that
\begin{align}
\lambda_{max}&\le \sup_{\substack{d_1,\dotsc,d_k\\ \prod_{i=1}^kd_i \text{ square-free}}}y_{max}\Bigl(\prod_{i=1}^kd_i\Bigr)\sum_{\substack{r_1,\dotsc,r_k \\d_i|r_i\forall i\\ \prod_{i=1}^k r_i<R\\ \prod_{i=1}^k r_i \text{ square-free}}}\Bigl(\prod_{i=1}^k\frac{\mu(r_i)^2}{\phi(r_i)}\Bigr)\nonumber\\
&\le y_{max} \sup_{\substack{d_1,\dotsc,d_k\\ \prod_{i=1}^kd_i \text{ square-free}}}\Bigl(\prod_{i=1}^k\frac{d_i}{\phi(d_i)}\Bigr)\sum_{\substack{r'<R/\prod_{i=1}^kd_i\\ (r',\prod_{i=1}^kd_i)=1}}\frac{\mu(r')^2\tau_k(r')}{\phi(r')}\nonumber\\
&\le y_{max} \sup_{d_1,\dotsc,d_k}\sum_{d|\prod_{i=1}^kd_i}\frac{\mu(d)^2}{\phi(d)}\sum_{\substack{r'<R/\prod_{i=1}^kd_i\\ (r',\prod_{i=1}^kd_i)=1}}\frac{\mu(r')^2\tau_k(r')}{\phi(r')}\nonumber\\
&\le y_{max}\sum_{u<R}\frac{\mu(u)^2\tau_k(u)}{\phi(u)}\ll y_{max}(\log{R})^k.\label{eq:LambdaSize}
\end{align} 
In the last line we have taken $u=dr'$, and used the fact $\tau_k(dr')\ge\tau_k(r')$. Hence the error term $O(\lambda_{max}^2R^2(\log{N})^{2k})$ is of size $O( y_{max}^2R^2(\log{N})^{4k})$.

Substituting our change of variables \eqref{eq:YLambdaDef} into the main term \eqref{eq:S1MainTerm}, and using the above estimate for the error term, we obtain
\begin{align}
S_1&=\frac{N}{W}\sum_{u_1,\dotsc,u_k}\Bigl(\prod_{i=1}^k\phi(u_i)\Bigr)\sideset{}{^*}\sum_{s_{1,2},\dotsc, s_{k,k-1}} \Bigl(\prod_{\substack{1\le i,j\le k\\i\ne j}}\mu(s_{i,j})\Bigr)\Bigl(\prod_{i=1}^k\frac{\mu(a_i)\mu(b_i)}{\phi(a_i)\phi(b_i)}\Bigr)y_{a_1,\dotsc,a_k}y_{b_1,\dotsc,b_k}\nonumber\\
&\qquad+O\Bigl(y_{max}^2R^2(\log{R})^{4k}\Bigr),
\end{align}
where $a_j=u_j\prod_{i\ne j}s_{j,i}$ and $b_j=u_j\prod_{i\ne j}s_{i,j}$. In these expressions we have used the fact that we have restricted $s_{i,j}$ to be coprime to the other terms in the expression for $a_i$ and $b_j$. For the same reason we may rewrite $\mu(a_j)$ as $\mu(u_j)\prod_{i\ne j}\mu(s_{i,j})$, and similarly for $\phi(a_j)$, $\mu(b_j)$ and $\phi(b_j)$. This gives us
\begin{align}
S_1&=\frac{N}{W}\sum_{u_1,\dotsc,u_k}\Bigl(\prod_{i=1}^k\frac{\mu(u_i)^2}{\phi(u_i)}\Bigr)\sideset{}{^*}\sum_{s_{1,2},\dotsc, s_{k,k-1}} \Bigl(\prod_{\substack{1\le i,j\le k\\i\ne j}}\frac{\mu(s_{i,j})}{\phi(s_{i,j})^2}\Bigr)y_{a_1,\dotsc,a_k}y_{b_1,\dotsc,b_k}+O\Bigl(y_{max}^2R^2(\log{R})^{4k}\Bigr).\label{eq:MultiplicativeSplit}
\end{align}
We see that there is no contribution from $s_{i,j}$ with $(s_{i,j},W)\ne 1$ because of the restricted support of $y$. Thus we only need to consider $s_{i,j}=1$ or $s_{i,j}>D_0$. The contribution when $s_{i,j}>D_0$ is
\begin{align}
&\ll \frac{y_{max}^2 N}{W}\Bigl(\sum_{\substack{u<R\\ (u,W)=1}}\frac{\mu(u)^2}{\phi(u)}\Bigr)^k\Bigl(\sum_{s_{i,j}>D_0}\frac{\mu(s_{i,j})^2}{\phi(s_{i,j})^2}\Bigr)\Bigl(\sum_{s\ge 1}\frac{\mu(s)^2}{\phi(s)^2}\Bigr)^{k^2-k-1}\ll \frac{y_{max}^2 \phi(W)^k N (\log{R})^k}{W^{k+1}D_0}.
\end{align}
Thus we may restrict our attention to the case when $s_{i,j}=1$ $\forall i\ne j$. This gives
\begin{equation}
S_1=\frac{N}{W}\sum_{u_1,\dotsc,u_k}\frac{y_{u_1,\dotsc,u_k}^2}{\prod_{i=1}^k\phi(u_i)}+O\left(\frac{y_{max}^2 \phi(W)^k N (\log{R})^k}{W^{k+1}D_0}+y_{max}^2R^2(\log{R})^{4k}\right).
\end{equation}
We recall that $R^2=N^{\theta-2\delta}\le N^{1-2\delta}$ and $W\ll N^{\delta}$, and so the first error term dominates. This gives the result.
\end{proof}
We now consider $S_2$. We write $S_2=\sum_{m=1}^kS_2^{(m)}$, where
\begin{equation}
S_2^{(m)}=\sum_{\substack{N\le n<2N\\ n\equiv v_0\pmod{W}}}\chi_{\mathbb{P}}(n+h_m)\Bigl(\sum_{\substack{d_1,\dotsc, d_k\\ d_i|n+h_i\forall i}}\lambda_{d_1,\dotsc,d_k}\Bigr)^2.
\end{equation}
We now estimate $S_2^{(m)}$ in a similar way to our treatment of $S_1$.
\begin{lmm}\label{lmm:S2Expression1}
Let
\[y^{(m)}_{r_1,\dotsc,r_k}=\Bigl(\prod_{i=1}^k\mu(r_i)g(r_i)\Bigr)\sum_{\substack{d_1,\dotsc,d_k\\ r_i|d_i\forall i\\ d_m=1}}\frac{\lambda_{d_1,\dotsc,d_k}}{\prod_{i=1}^k \phi(d_i)},\]
where $g$ is the totally multiplicative function defined on primes by $g(p)=p-2$. Let $y^{(m)}_{max}=\sup_{r_1,\dotsc,r_k}|y^{(m)}_{r_1,\dotsc,r_k}|$. Then for any fixed $A>0$ we have
\[S_2^{(m)}=\frac{N}{\phi(W)\log{N}}\sum_{r_1,\dotsc,r_k}\frac{(y^{(m)}_{r_1,\dotsc,r_k})^2}{\prod_{i=1}^kg(r_i)}+O\Bigl(\frac{(y^{(m)}_{max})^2 \phi(W)^{k-2} N (\log{N})^{k-2}}{W^{k-1}D_0}\Bigr)+O\Bigl(\frac{y_{max}^2 N}{(\log{N})^A}\Bigr).\]
\end{lmm}
\begin{proof}
We first expand out the square and swap the order of summation to give
\begin{equation}
S_2^{(m)}=\sum_{\substack{d_1,\dotsc,d_k\\ e_1,\dotsc,e_k}}\lambda_{d_1,\dotsc,d_k}\lambda_{e_1,\dotsc,e_k}\sum_{\substack{N\le n<2N\\ n\equiv v_0\pmod{W}\\ [d_i,e_i]|n+h_i\forall i}}\chi_\mathbb{P}(n+h_m).
\end{equation}
As with $S_1$, the inner sum can be written as a sum over a single residue class modulo $q=W\prod_{i=1}^k[d_i,e_i]$, provided that $W,[d_1,e_1],\dotsc, [d_k,e_k]$ are pairwise coprime. The integer $n+h_m$ will lie in a residue class coprime to the modulus if and only if $d_m=e_m=1$. In this case the inner sum will contribute $X_N/\phi(q)+O(E(N,q))$, where
\begin{align}
E(N,q)&=1+\sup_{(a,q)=1}\Bigl|\sum_{\substack{N\le n<2N\\ n\equiv a \pmod{q}}}\chi_{\mathbb{P}}(n)-\frac{1}{\phi(q)}\sum_{N\le n<2N}\chi_\mathbb{P}(n)\Bigr|,\\
X_N&=\sum_{N\le n<2N}\chi_{\mathbb{P}}(n).\end{align}
If either one pair of $W, [d_1,e_1],\dotsc, [d_k,e_k]$ share a common factor, or if either $d_m$ or $e_m$ are not $1$, then the contribution of the inner sum is zero. Thus we obtain
\begin{equation}
S_2^{(m)}=\frac{X_N}{\phi(W)}\sideset{}{'}\sum_{\substack{d_1,\dotsc,d_k\\e_1,\dotsc,e_k\\ e_m=d_m=1}}\frac{\lambda_{d_1,\dotsc,d_k}\lambda_{e_1,\dotsc,e_k}}{\prod_{i=1}^k\phi([d_i,e_i])}+O\Bigl(\sum_{\substack{d_1,\dotsc,d_k\\ e_1,\dotsc,e_k}}|\lambda_{d_1,\dotsc,d_k}\lambda_{e_1,\dotsc,e_k}|E(N,q)\Bigr),
\end{equation}
where we have written $q=W\prod_{i=1}^k[d_i,e_i]$.

We first deal with the contribution from the error terms. From the support of $\lambda_{d_1,\dotsc,d_k}$, we see that we only need to consider square-free $q$ with $q<R^2W$. Given a square-free integer $r$, there are at most $\tau_{3k}(r)$ choices of $d_1,\dotsc,d_k,e_1,\dotsc, e_k$ for which $W\prod_{i=1}^k[d_i,e_i]=r$. We also recall from \eqref{eq:LambdaSize} that $\lambda_{max}\ll y_{max}(\log{R})^k$. Thus the error term contributes
\begin{equation}
\ll y_{max}^2(\log{R})^{2k}\sum_{r<R^2W}\mu(r)^2\tau_{3k}(r)E(N,r).
\end{equation}
By Cauchy-Schwarz, the trivial bound $E(N,q)\ll N/\phi(q)$, and our hypothesis that the primes have level of distribution $\theta$, this contributes for any fixed $A>0$
\begin{equation}
\ll y_{max}^2(\log{R})^{2k}\Bigl(\sum_{r<R^2W}\mu(r)^2\tau^2_{3k}(r)\frac{N}{\phi(r)}\Bigr)^{1/2}\Bigl(\sum_{r<R^2W}\mu(r)^2E(N,r)\Bigr)^{1/2}\ll \frac{y_{max}^2N}{(\log{N})^A}.
\end{equation}
We now concentrate on the main sum. As in the treatment of $S_1$ in the proof of Lemma \ref{lmm:S1Expression1}, we rewrite the conditions $(d_{i},e_{j})=1$ by multiplying our expression by $\sum_{s_{i,j}|d_i,e_j}\mu(s_{i,j})$. Again we may restrict $s_{i,j}$ to be coprime to $u_i$, $u_j$, $s_{i,a}$ and $s_{b,j}$ for all $a\ne j$ and $b\ne i$. We denote the summation subject to these restrictions by $\sum^*$. We also split the $\phi([d_i,e_i])$ terms by using the equation (valid for square-free $d_i,e_i$)
\begin{equation}
\frac{1}{\phi([d_i,e_i])}=\frac{1}{\phi(d_i)\phi(e_i)}\sum_{u_i|d_i,e_i}g(u_i),
\end{equation}
where $g$ is the totally multiplicative function defined on primes by $g(p)=p-2$. This gives us
a main term of
\begin{equation}
\frac{X_N}{\phi(W)}\sum_{u_1,\dotsc,u_k}\Bigl(\prod_{i=1}^k g(u_i)\Bigr)\sideset{}{^*}\sum_{s_{1,2},\dotsc, s_{k,k-1}}\Bigl(\prod_{\substack{1\le i,j\le k\\ i\ne j}}\mu(s_{i,j})\Bigr)\sum_{\substack{d_1,\dotsc,d_k\\e_1,\dotsc,e_k\\ u_i|d_i,e_i\forall i\\ s_{i,j}|d_{i},e_{j}\forall i\ne j\\ d_m=e_m=1}}\frac{\lambda_{d_1,\dotsc,d_k}\lambda_{e_1,\dotsc,e_k}}{\prod_{i=1}^k\phi(d_i)\phi(e_i)}.\label{eq:S2MainTerm}
\end{equation}
We have now separated the dependencies between the $e$ and $d$ variables, so again we make a substitution. We let
\begin{equation}
y^{(m)}_{r_1,\dotsc,r_k}=\Bigl(\prod_{i=1}^k \mu(r_i)g(r_i)\Bigr)\sum_{\substack{d_1,\dotsc, d_{k}\\ r_i|d_i\forall i\\ d_m=1}}\frac{\lambda_{d_1,\dotsc,d_k}}{\prod_{i=1}^k\phi(d_i)}.\label{eq:YjDef}
\end{equation}
We note $y^{(m)}_{r_1,\dotsc,r_k}=0$ unless $r_m=1$. Substituting this into \eqref{eq:S2MainTerm}, we obtain a main term of
\begin{equation}
\frac{X_N}{\phi(W)}\sum_{u_1,\dotsc,u_k}\Bigl(\prod_{i=1}^k\frac{\mu(u_i)^2}{g(u_i)}\Bigr)\sideset{}{^*}\sum_{s_{1,2},\dotsc,s_{k,k-1}}\Bigl(\prod_{\substack{1\le i,j\le k\\ i\ne j}}\frac{\mu(s_{i,j})}{g(s_{i,j})^2}\Bigr)y^{(m)}_{a_1,\dotsc,a_k}y^{(m)}_{b_1,\dotsc,b_k},
\end{equation}
where $a_j=u_j\prod_{i\ne j}s_{j,i}$ and $b_j=u_j\prod_{i\ne j}s_{i,j}$ for each $1\le j\le k$. As before, we have replaced  $\mu(a_j)$ with $\mu(u_j)\prod_{i\ne j}\mu(s_{j,i})$ (and similarly for $g(a_j)$, $\mu(b_j)$ and $g(b_j)$). This is valid since terms with $a_j$ or $b_j$ not square-free make no contribution.

We see the contribution from $s_{i,j}\ne 1$ is of size
\begin{align}
&\ll \frac{(y^{(m)}_{max})^2 N}{\phi(W)\log{N}}\Bigr(\sum_{\substack{u<R\\(u,W)=1}}\frac{\mu(u)^2}{g(u)}\Bigr)^{k-1}\Bigl(\sum_{s}\frac{\mu(s)^2}{g(s)^2}\Bigr)^{k(k-1)-1}\sum_{s_{i,j}>D_0}\frac{\mu(s_{i,j})^2}{g(s_{i,j})^2}\nonumber\\
&\ll \frac{(y^{(m)}_{max})^2 \phi(W)^{k-2} N (\log{R})^{k-1}}{W^{k-1} D_0 \log{N}}.
\end{align}
Thus we find that
\begin{equation}
S_2^{(m)}=\frac{X_N}{\phi(W)}\sum_{u_1,\dotsc,u_k}\frac{(y^{(m)}_{u_1,\dotsc,u_k})^2}{\prod_{i=1}^kg(u_i)}+O\Bigl(\frac{(y_{max}^{(m)})^2 \phi(W)^{k-2} N (\log{R})^{k-2}}{D_0W^{k-1}}\bigr)+O\Bigl(\frac{y_{max}^2N}{(\log{N})^A}\Bigr).\label{eq:SmAlmostFinished}
\end{equation}
Finally, by the prime number theorem, $X_N=N/\log{N}+O(N/(\log{N})^2)$. This error term contributes
\begin{equation}
\ll \frac{(y^{(m)}_{max})^2 N}{\phi(W)(\log{N})^2}\Bigl(\sum_{\substack{u<R\\ (u,W)=1}}\frac{\mu(u)^2}{g(u)}\Bigr)^{k-1}\ll \frac{(y^{(m)}_{max})^2 \phi(W)^{k-2} N (\log{R})^{k-3}}{W^{k-1}},
\end{equation}
which can be absorbed into the first error term of \eqref{eq:SmAlmostFinished}. This completes the proof.
\end{proof}
\begin{rmk}
In our proof of Lemma \ref{lmm:S2Expression1} we only really require $\lambda_{d_1,\dotsc,d_k}$ to be supported on $d_1,\dotsc,d_k$ satisfying $\prod_{i\ne j}d_i<R$ for all $j$ instead of $\prod_{i=1}^kd_i<R$. For $k\ge 3$, the numerical benefit of this extension is small and so we do not consider it further.
\end{rmk}
\begin{rmk}
As our result relies on the Bombieri-Vinogradov theorem, the implied constant in the error term is not effectively computable. However, if we restrict the $\lambda_{d_1,\dotsc,d_k}$ to be supported on $d_i$ which are coprime to the largest prime factor of a possible exceptional modulus of a primitive character then we can make this error term (and all others in this paper) effective at the cost of a negligible error.
\end{rmk}
We now relate our new variables $y^{(m)}_{r_1,\dotsc,r_{k}}$ to the $y_{r_1,\dotsc,r_k}$ variables from $S_1$.
\begin{lmm}\label{lmm:YjExpression}
If $r_m=1$ then
\[y^{(m)}_{r_1,\dotsc,r_k}=\sum_{a_m}\frac{y_{r_1,\dotsc,r_{m-1},a_m,r_{m+1},\dotsc,r_k}}{\phi(a_m)}+O\Bigl(\frac{y_{max} \phi(W) \log{R}}{W D_0}\Bigr).\]
\end{lmm}
\begin{proof}
We assume throughout the proof that $r_m=1$. We first substitute our expression \eqref{eq:LambdaYDef} into the definition \eqref{eq:YjDef}. This gives
\begin{equation}
y^{(m)}_{r_1,\dotsc,r_k}=\Bigl(\prod_{i=1}^k\mu(r_i)g(r_i)\Bigr)\sum_{\substack{d_1,\dotsc,d_k\\r_i|d_i\forall i\\d_m=1}}\Bigl(\prod_{i=1}^k\frac{\mu(d_i)d_i}{\phi(d_i)}\Bigr)\sum_{\substack{a_1,\dotsc,a_k\\d_i|a_i\forall i}}\frac{y_{a_1,\dotsc,a_k}}{\prod_{i=1}^k\phi(a_i)}.
\end{equation}
We swap the summation of the $d$ and $a$ variables to give
\begin{equation}
y^{(m)}_{r_1,\dotsc,r_k}=\Bigl(\prod_{i=1}^k\mu(r_i)g(r_i)\Bigr)\sum_{\substack{a_1,\dotsc,a_k\\r_i|a_i\forall i}}\frac{y_{a_1,\dotsc,a_k}}{\prod_{i=1}^k\phi(a_i)}\sum_{\substack{d_1,\dotsc, d_k\\ d_i|a_i, r_i|d_i\forall i\\ d_m=1}}\prod_{i=1}^k\frac{\mu(d_i)d_i}{\phi(d_i)}.
\end{equation}
We can now evaluate the sum over $d_1,\dotsc,d_k$ explicitly. This gives
\begin{align}
y^{(m)}_{r_1,\dotsc,r_k}&=\Bigl(\prod_{i=1}^k\mu(r_i)g(r_i)\Bigr)\sum_{\substack{a_1,\dotsc,a_k\\r_i|a_i\forall i}}\frac{y_{a_1,\dotsc,a_k}}{\prod_{i=1}^k\phi(a_i)}\prod_{i\ne m}\frac{\mu(a_i)r_i}{\phi(a_i)}.
\end{align}
We see that from the support of $y_{a_1,\dotsc,a_k}$ that we may restrict the summation over $a_j$ to $(a_j,W)=1$. Thus either $a_j=r_j$ or $a_j>D_0r_j$. For $j\ne m$, the total contribution from $a_{j}\ne r_{j}$ is
\begin{align}
&\ll y_{max}\Bigl(\prod_{i=1}^kg(r_i)r_i\Bigr) \Bigl(\sum_{\substack{a_{j}>D_0r_{j}\\ r_j|a_j}}\frac{\mu(a_{j})^2}{\phi(a_{j})^2}\Bigr)\Bigl(\sum_{\substack{a_m<R\\ (a_m,W)=1}}\frac{\mu(a_m)^2}{\phi(a_m)}\Bigr)\prod_{\substack{1\le i\le k\\ i\ne j,m}}\Bigl(\sum_{r_i|a_i}\frac{\mu(a_i)^2}{\phi(a_i)^2}\Bigr)\nonumber\\
&\ll \Bigl(\prod_{i=1}^k \frac{g(r_i)r_i}{\phi(r_i)^2}\Bigr)\frac{y_{max}\phi(W)\log{R}}{W D_0}\ll \frac{y_{max} \phi(W) \log{R}}{W D_0}.
\end{align}
Thus we find that the main contribution is when $a_j=r_j$ for all $j\ne m$. We have
\begin{align}
y^{(m)}_{r_1,\dotsc,r_k}=\Bigl(\prod_{i=1}^k\frac{g(r_i)r_i}{\phi(r_i)^2}\Bigr)\sum_{a_m}\frac{y_{r_1,\dotsc, r_{m-1},a_m,r_{m+1},\dotsc,r_k}}{\phi(a_m)}+O\Bigl(\frac{y_{max} \phi(W) \log{R}}{W D_0}\Bigr). 
\end{align}
We note that $g(p)p/\phi(p)^2=1+O(p^{-2})$. Thus, since the contribution is zero unless $\prod_{i=1}^k r_i$ is coprime to $W$, we see that the product in the above expression may be replaced by $1+O(D_0^{-1})$. This gives the result.
\end{proof}
\section{\texorpdfstring{Smooth choice of $y$}{Smooth choice of y}}
We now choose suitable values for our $y$ variables, and complete the proof of Proposition \ref{prpstn:MainProp}.

We first give some comments to motivate our choice of the $y$ variables, which we believe should be close to optimal. We wish to choose $y$ so as to maximize the ratio of the main terms of $S_2$ and $S_1$. If we use Lagrangian multipliers to maximize this ratio (treating all error terms as zero) we arrive at the condition that
\begin{equation}
\lambda y_{r_1,\dotsc,r_k}=\Bigl(\prod_{i=1}^k \frac{\phi(r_i)}{g(r_i)}\Bigr)\sum_{m=1}^k\frac{g(r_m)}{\phi(r_m)}y^{(m)}_{r_1,\dotsc,r_{m-1},1,r_{m+1},\dotsc,r_k}
\end{equation}
for some fixed constant $\lambda$. The $y$ terms are supported on integers free of small prime factors, and for most integers $r$ free of small prime factors we have $g(r)\approx\phi(r)\approx r$, and so the above condition reduces to
\begin{equation}
\lambda y_{r_1,\dotsc,r_k}\approx\sum_{m=1}^{k}y^{(m)}_{r_1,\dotsc,r_{m-1},1,r_{m+1},\dotsc,r_k}.
\end{equation}
This condition looks smooth (it has no dependence on the prime factorization of the $r_i$), and should be able to be satisfied if $y_{r_1,\dotsc,r_k}$ is a smooth function of the $r_i$ variables. Motivated by the above, when the product $r=\prod_{i=1}^k r_i$ satisfies $(r,W)=1$ and $\mu(r)^2=1$ we choose
\begin{equation}
y_{r_1,\dotsc,r_k}=F\Bigl(\frac{\log{r_1}}{\log{R}},\dotsc,\frac{\log{r_k}}{\log{R}}\Bigr),\label{eq:YChoice}
\end{equation}
for some smooth function $F:\mathbb{R}^k\rightarrow\mathbb{R}$, supported on $\mathcal{R}_k=\{(x_1,\dotsc,x_k)\in[0,1]^k:\sum_{i=1}^kx_i\le 1\}$. As previously required, we set $y_{r_1,\dotsc,r_k}=0$ if the product $r$ is either not coprime to $W$ or is not square-free. With this choice of $y$, we can obtain suitable asymptotic estimates for $S_1$ and $S_2$.

We will use the following Lemma to estimate our sums $S_1$ and $S_2$ with this choice of $y$.
\begin{lmm}\label{lmm:PartialSummation}
Let $A_1,A_2,L>0$. Let $\gamma$ be a multiplicative function satisfying
\[0\le \frac{\gamma(p)}{p}\le 1-A_1,\]
and
\[-L\le \sum_{w\le p\le z}\frac{\gamma(p)\log{p}}{p}- \log{z/w}\le A_2\]
for any $2\le w\le z$. Let $g$ be the totally multiplicative function defined on primes by $g(p)=\gamma(p)/(p-\gamma(p))$. Finally, let $G:[0,1]\rightarrow \mathbb{R}$ be smooth, and let $G_{max}=\sup_{t\in [0,1]}(|G(t)|+|G'(t)|)$. Then
\[\sum_{d<z}\mu(d)^2g(d)G\Bigl(\frac{\log{d}}{\log{z}}\Bigr)=\mathfrak{S}\log{z}\int_0^1G(x)dx+O_{A_1,A_2}(\mathfrak{S}LG_{max}),\]
where
\[\mathfrak{S}=\prod_{p}\Bigl(1-\frac{\gamma(p)}{p}\Bigr)^{-1}\Bigl(1-\frac{1}{p}\Bigr).\]
Here the constant implied by the `$O$' term is independent of $G$ and $L$.
\end{lmm}
\begin{proof}
This is \cite[Lemma 4]{GGPY}, with $\kappa=1$ and slight changes to the notation.
\end{proof}
We now finish our estimations of $S_1$ and $S_2^{(m)}$, completing  the proof of Proposition \ref{prpstn:MainProp}. We first estimate $S_1$.
\begin{lmm}\label{lmm:S1Summation2}
Let $y_{r_1,\dotsc,r_k}$ be given in terms of a smooth function $F$ by \eqref{eq:YChoice}, with $F$ supported on $\mathcal{R}_k=\{(x_1,\dotsc,x_k)\in[0,1]^k:\sum_{i=1}^kx_i\le 1\}$. Let 
\[F_{max}=\sup_{(t_1,\dotsc,t_k)\in[0,1]^k}|F(t_1,\dotsc,t_k)|+\sum_{i=1}^k|\frac{\partial F}{\partial t_i}(t_1,\dotsc,t_k)|. \]
Then we have
\[S_1=\frac{\phi(W)^k N (\log{R})^k}{W^{k+1}}I_k(F)+O\Bigl(\frac{F_{max}^2 \phi(W)^{k} N (\log{R})^k}{W^{k+1} D_0}\Bigr),\]
where
\[I_k(F)=\int_0^1\dotsi\int_0^1F(t_1,\dotsc,t_k)^2dt_1\dotsc dt_k.\]
\end{lmm}
\begin{proof}
We substitute our choice \eqref{eq:YChoice} of $y$ into our expression of $S_1$ in terms of $y_{r_1,\dotsc,r_k}$ given by Lemma \ref{lmm:S1Expression1}. This gives
\begin{equation}
S_1=\frac{N}{W}\sum_{\substack{u_1,\dotsc,u_k\\ (u_i,u_j)=1\forall i\ne j\\ (u_i,W)=1\forall i}}\Bigl(\prod_{i=1}^{k}\frac{\mu(u_i)^2}{\phi(u_i)}\Bigr)F\Bigl(\frac{\log{u_1}}{\log{R}},\dotsc,\frac{\log{u_k}}{\log{R}}\Bigr)^2+O\Bigl(\frac{F_{max}^2 \phi(W)^k N (\log{R})^k}{W^{k+1} D_0}\Bigr).\label{eq:S1BasicExpression}
\end{equation}
We note that two integers $a$ and $b$ with $(a,W)=(b,W)=1$ but $(a,b)\ne 1$ must have a common prime factor which is greater than $D_0$. Thus we can drop the requirement that $(u_{i},u_{j})=1$, at the cost of an error of size
\begin{align}
&\ll \frac{F_{max}^2 N}{W}\sum_{p>D_0}\sum_{\substack{u_1,\dotsc,u_k<R\\ p|u_{i},u_{j}\\ (u_i,W)=1\forall i}}\prod_{i=1}^k\frac{\mu(u_i)^2}{\phi(u_i)}\nonumber\\
&\ll \frac{F_{max}^2 N}{W}\sum_{p>D_0}\frac{1}{(p-1)^2}\Bigl(\sum_{\substack{u<R\\ (u,W)=1}}\frac{\mu(u)^2}{\phi(u)}\Bigr)^k\ll \frac{F_{max}^2 \phi(W)^k N (\log{R})^k}{W^{k+1} D_0}.\label{eq:S1CoprimeError}
\end{align}
Thus we are left to evaluate the sum
\begin{equation}
\sum_{\substack{u_1,\dotsc,u_k\\ (u_i,W)=1\forall i}}\Bigl(\prod_{i=1}^{k}\frac{\mu(u_i)^2}{\phi(u_i)}\Bigr)F\Bigl(\frac{\log{u_1}}{\log{R}},\dotsc,\frac{\log{u_k}}{\log{R}}\Bigr)^2.
\end{equation}
We can now estimate this sum by $k$ applications of Lemma \ref{lmm:PartialSummation}, dealing with the sum over each $u_i$ in turn. For each application we take
\begin{align}
\gamma(p)&=\begin{cases}1,\qquad &p\nmid W,\\
0,&\text{otherwise,}\end{cases}\\
L&\ll 1+\sum_{p|W}\frac{\log{p}}{p}\ll \log{D_0},
\end{align}
and $A_1$ and $A_2$ fixed constants of suitable size. This gives
\begin{align}
\sum_{\substack{u_1,\dotsc,u_k\\ (u_i,W)=1\forall i}}\Bigl(\prod_{i=1}^{k}\frac{\mu(u_i)^2}{\phi(u_i)}\Bigr)F\Bigl(\frac{\log{u_1}}{\log{R}},\dotsc,\frac{\log{u_k}}{\log{R}}\Bigr)^2&=\frac{\phi(W)^k (\log{R})^k}{W^k}I_k(F)\nonumber\\
&\qquad+O\Bigl(\frac{F_{max}^2\phi(W)^k(\log{D_0})(\log{R})^{k-1}}{W^k}\Bigr).\label{eq:S1SmoothSum}
\end{align}
We now combine \eqref{eq:S1SmoothSum} with \eqref{eq:S1BasicExpression} and \eqref{eq:S1CoprimeError} to obtain the result.
\end{proof}

\begin{lmm}
Let $y_{r_1,\dotsc,r_k}$, $F$ and $F_{max}$ be as described in Lemma \ref{lmm:S1Summation2}. Then we have
\[S_2^{(m)}=\frac{\phi(W)^k N (\log{R})^{k+1}}{W^{k+1}\log{N}}J^{(m)}_k(F)+O\Bigl(\frac{F_{max}^2\phi(W)^{k}N(\log{R})^k}{W^{k+1} D_0}\Bigr),\]
where
\[J^{(m)}_k(F)=\int_0^1\dotsi\int_0^1\Bigr(\int_0^1F(t_1,\dotsc,t_k)dt_m\Bigr)^2dt_1\dotsc dt_{m-1}dt_{m+1}\dotsc dt_k.\]
\end{lmm}
\begin{proof}
The estimation of $S_2^{(m)}$ is similar to the estimation of $S_1$. We first estimate $y^{(m)}_{r_1,\dotsc,r_k}$. We recall that $y^{(m)}_{r_1,\dotsc,r_k}=0$ unless $r_m=1$ and $r=\prod_{i=1}^kr_i$ satisfies $(r,W)=1$ and $\mu(r)^2=1$, in which case $y^{(m)}_{r_1,\dotsc,r_k}$ is given in terms of $y_{r_1,\dotsc,r_k}$ by Lemma \ref{lmm:YjExpression}. We first concentrate on this case when $y^{(m)}_{r_1,\dotsc,r_k}\ne 0$. We substitute our choice \eqref{eq:YChoice} of $y$ into our expression from Lemma \ref{lmm:YjExpression}. This gives
\begin{align}
y^{(m)}_{r_1,\dotsc,r_k}&=\sum_{(u,W\prod_{i=1}^k r_i)=1}\frac{\mu(u)^2}{\phi(u)}F\Bigl(\frac{\log{r_1}}{\log{R}},\dotsc,\frac{\log{r_{m-1}}}{\log{R}},\frac{\log{u}}{\log{R}},\frac{\log{r_{m+1}}}{\log{R}},\dotsc,\frac{\log{r_{k}}}{\log{R}}\Bigr)\nonumber\\
&\qquad+O\Bigl(\frac{F_{max} \phi(W) \log{R}}{W D_0}\Bigr).\label{eq:YmFExpression}\end{align}
We can see from this that $y^{(m)}_{max}\ll \phi(W)F_{max}(\log{R})/W$. We now estimate the sum over $u$ in \eqref{eq:YmFExpression}. We apply Lemma \ref{lmm:PartialSummation} with
\begin{align}
\gamma(p)&=\begin{cases}1,\qquad &p\nmid W\prod_{i=1}^k r_i,\\
0,&\text{otherwise,}\end{cases}\\
L&\ll 1+\sum_{p|W\prod_{i=1}^k r_i}\frac{\log{p}}{p}\ll \sum_{p<\log{R}}\frac{\log{p}}{p}+\sum_{\substack{p|W\prod_{i=1}^k r_i \\ p>\log{R}}}\frac{\log{\log{R}}}{\log{R}}\ll \log{\log{N}},
\end{align}
and with $A_1,A_2$ suitable fixed constants. This gives us
\begin{align}
y^{(m)}_{r_1,\dotsc,r_k}&=(\log{R})\frac{\phi(W)}{W}\bigl(\prod_{i=1}^k \frac{\phi(r_i)}{r_i}\bigr)F^{(m)}_{r_1,\dotsc,r_{k}}+O\Bigl(\frac{F_{max} \phi(W)\log{R}}{WD_0}\Bigr),\label{eq:YjFExpression}\end{align}
where
\begin{equation}
F^{(m)}_{r_1,\dotsc,r_k}=\int_0^1F\Bigl(\frac{\log{r_1}}{\log{R}},\dotsc,\frac{\log{r_{m-1}}}{\log{R}},t_m,\frac{\log{r_{m+1}}}{\log{R}},\dotsc,\frac{\log{r_k}}{\log{R}}\Bigr)dt_m.
\end{equation}
Thus we have shown that if $r_m=1$ and $r=\prod_{i=1}^kr_i$ satisfies $(r,W)=1$ and $\mu(r)^2=1$ then $y^{(m)}_{r_1,\dotsc,r_k}$ is given by \eqref{eq:YjFExpression}, and otherwise $y^{(m)}_{r_1,\dotsc,r_k}=0$. We now substitute this into our expression from Lemma \ref{lmm:S2Expression1}, namely
\begin{equation}
S_2^{(m)}=\frac{N}{\phi(W)\log{N}}\sum_{r_1,\dotsc,r_k}\frac{(y^{(m)}_{r_1,\dotsc,r_k})^2}{\prod_{i=1}^kg(r_i)}+O\Bigl(\frac{(y^{(m)}_{max})^2 \phi(W)^{k-2} N (\log{N})^{k-2}}{W^{k-1} D_0}\Bigr)+O\Bigl(\frac{y_{max}^2 N}{(\log{N})^A}\Bigr).\end{equation}
 We obtain
\begin{align}
S_2^{(m)}&=\frac{\phi(W) N (\log{R})^2}{W^2\log{N}}\sum_{\substack{r_1,\dotsc,r_k\\(r_i,W)=1\forall i\\(r_{i},r_{j})=1 \forall i\ne j\\ r_m=1}}\Bigl(\prod_{i=1}^k\frac{\mu(r_i)^2\phi(r_i)^2}{g(r_i)r_i^2}\Bigr)(F^{(m)}_{r_1,\dotsc,r_k})^2+O\Bigl(\frac{F_{max}^2 \phi(W)^{k} N (\log{R})^k}{W^{k+1} D_0}\Bigl).
\end{align}
We remove the condition that $(r_{i},r_{j})=1$ in the same way we did when considering $S_1$. Instead of \eqref{eq:S1CoprimeError}, this introduces an error which is of size
\begin{align}
&\ll \frac{\phi(W) N (\log{R})^2F_{max}^2}{W^2\log{N}}\Bigl(\sum_{p>D_0}\frac{\phi(p)^4}{g(p)^2p^4}\Bigr)\Bigl(\sum_{\substack{r<R\\ (r,W)=1}}\frac{\mu(r)^2\phi(r)^2}{g(r)r^2}\Bigr)^{k-1}\ll \frac{F_{max}^2 \phi(W)^k N (\log{N})^k}{W^{k+1} D_0}.
\end{align}
Thus we are left to evaluate the sum
\begin{equation}
\sum_{\substack{r_1,\dotsc,r_{m-1},r_{m+1},\dotsc,r_k\\(r_i,W)=1\forall i }}\Bigl(\prod_{\substack{1\le i \le k\\i\ne j}}\frac{\mu(r_i)^2\phi(r_i)^2}{g(r_i)r_i^2}\Bigr)(F^{(m)}_{r_1,\dotsc,r_k})^2.
\end{equation}
We estimate this by applying Lemma \ref{lmm:PartialSummation} to each summation variable in turn. In each case we take
\begin{align}
\gamma(p)&=\begin{cases}1-\frac{p^2-3p+1}{p^3-p^2-2p+1},\qquad&p\nmid W\\ 0,&\text{ otherwise,}\end{cases}\\
L&\ll 1+\sum_{p|W}\frac{\log{p}}{p}\ll \log{D_0},
\end{align}
and $A_1,A_2$ suitable fixed constants. This gives
\begin{equation}
S_2^{(m)}=\frac{\phi(W)^{k} N (\log{R})^{k+1}}{W^{k+1}\log{N}}J^{(m)}_k+O\Bigl(\frac{F_{max}^2 \phi(W)^k N (\log{N})^k}{W^{k+1}D_0}\Bigr),
\end{equation}
where
\begin{equation}
J^{(m)}_k=\int_0^1\dotsi\int_0^1\Bigl(\int_0^1 F(t_1,\dotsc,t_k)dt_m\Bigr)^2dt_1\dotsc dt_{m-1}dt_{m+1}\dotsc dt_k,
\end{equation}
as required.
\end{proof}
\begin{rmk}If $F(t_1,\dotsc,t_k)=G(\sum_{i=1}^kt_i)$ for some function $G$, then $I_k(F)$ and $J_k^{(m)}(F)$ simplify to $I_k(F)=\int_0^1G(t)^2t^{k-1}dt/(k-1)!$ and $J_k^{(m)}(F)=\int_0^1 (\int_t^1G(v)dv)^2 t^{k-2}dt/(k-2)!$ for each $m$, which is equivalent to the results obtained using the original GPY method using weights given by \eqref{eq:GPYWeights}.
\end{rmk}
\begin{rmk}
Tao gives an alternative approach to arrive at his equivalent of Proposition \ref{prpstn:MainProp}. His approach is to define $\lambda_{d_1,\dots,d_k}$ in terms of a suitable smooth function $f(t_1,\dotsc,t_k)$ as in \eqref{eq:ApproximateLambda}. He then estimates the corresponding sums directly using Fourier integrals. This is somewhat similar to the original paper of Goldston, Pintz and Y\i ld\i r\i m \cite{GPY}. Our function $F$ corresponds to $f(t_1,\dotsc,t_k)$ differentiated with respect to each coordinate.
\end{rmk}
\section{\texorpdfstring{Choice of smooth weight for large $k$}{Choice of smooth weight for large k}}
In this section we establish part $(3)$ of Proposition \ref{prpstn:FChoices}. Our argument here is closely related to that of Tao, who uses a probability theory proof.

We let $\mathcal{S}_k$ denote the set of Riemann-integrable functions $F:[0,1]^k\rightarrow\mathbb{R}$ supported on $\mathcal{R}_k=\{(x_1,\dotsc,x_k)\in[0,1]^k:\sum_{i=1}^kx_i\le 1\}$ with $I_k(F)\ne 0$ and $J_k^{(m)}(F)\ne 0$ for each $m$. We would like to obtain a lower bound for
\begin{equation}
M_k=\sup_{F\in\mathcal{S}_k}\frac{\sum_{m=1}^k J_k^{(m)}(F)}{I_k(F)}.
\end{equation}
\begin{rmk}
Let $\mathcal{L}_k$ denote the linear operator defined by
\begin{equation}
\mathcal{L}_kF(u_1,\dotsc,u_k)=\sum_{m=1}^k\int_0^{1-\sum_{i\ne m}u_i}F(u_1,\dotsc,u_{m-1},t_m,u_{m+1},\dotsc,u_k)dt_m
\end{equation}
whenever $(u_1,\dotsc,u_k)\in\mathcal{R}_k$, and zero otherwise. We expect that if $F$ maximizes the ratio $\sum_{m=1}^kJ_k^{(m)}(F)/I_k(F)$, then $F$ is an eigenfunction for $\mathcal{L}_k$, and the corresponding eigenvalue is the value of ratio at $F$. Unfortunately the author has not been able to solve the eigenvalue equation for $\mathcal{L}_k$ when $k>2$.
\end{rmk}
We obtain a lower bound for $M_k$ by constructing a function $F=F_k$ which makes the ratio $\sum_{m=1}^kJ_k^{(m)}(F)/I_k(F)$ large provided $k$ is large. We choose $F$ to be of the form
\begin{equation}\label{eq:FChoiceKLarge}
F(t_1,\dots,t_k)=\begin{cases}
\prod_{i=1}^k g(kt_i),\qquad &\text{if $\sum_{i=1}^kt_i\le 1$,}\\
0,&\text{otherwise,}\end{cases}
\end{equation}
for some smooth function $g:[0,\infty]\rightarrow\mathbb{R}$, supported on $[0,T]$. We see that with this choice $F$ is symmetric, and so $J^{(m)}_k(F)$ is independent of $m$. Thus we only need to consider $J_k=J_k^{(1)}(F)$. Similarly we write $I_k=I_k(F)$.

The key observation is that if the center of mass $\int_0^\infty u g(u)^2du/\int_0^\infty g(u)^2du$ of $g^2$ is strictly less than 1, then for large $k$ we expect that the constraints $\sum_{i=1}^kt_i\le 1$ to be able to be dropped at the cost of only a small error. This is because (by concentration of measure) the main contribution to the unrestricted integrals $I'_k=\int_0^\infty\dotsi\int_0^\infty \prod_{i=1}^kg(kt_i)^2 dt_1\dotsc dt_k$ and $J'_k=\int_0^\infty\dotsi\int_0^\infty(\int_0^\infty \prod_{i=1}^kg(kt_i) dt_1)^2dt_2\dotsc dt_k$ should come primarily from when $\sum_{i=1}^kt_i$ is close to the center of mass. Therefore  we would expect the contribution when $\sum_{i=1}^k t_i>1$ to be small if the center of mass is less than 1, and so $I_k$ and $J_k$ are well approximated by $I_k'$ and $J_k'$ in this case.

To ease notation we let $\gamma=\int_{u\ge 0}g(u)^2du$, and restrict our attention to $g$ such that $\gamma>0$. We have
\begin{equation}
I_k=\idotsint\limits_{\mathcal{R}_k}F(t_1,\dotsc,t_k)^2 dt_1\dotsc dt_k\le \Bigl(\int_0^\infty g(kt)^2dt\Bigr)^k=k^{-k}\gamma^k.\label{eq:IkBound}
\end{equation}
We now consider $J_k$. Since squares are non-negative, we obtain a lower bound for $J_k$ if we restrict the outer integral to $\sum_{i=2}^kt_i<1-T/k$. This has the advantage that, by the support of $g$, there are no further restrictions on the inner integral. Thus
\begin{equation}
J_k\ge \idotsint\limits_{\substack{t_2,\dotsc,t_k\ge 0\\ \sum_{i=2}^k t_i\le 1-T/k}}\Bigl(\int_0^{T/k}\Bigl(\prod_{i=1}^kg(kt_i)\Bigr)dt_1\Bigr)^2dt_2\dotsc dt_k.\label{eq:JkBound}
\end{equation}
We write the right hand side of \eqref{eq:JkBound} as $J'_k-E_k$, where
\begin{align}
J'_k&=\idotsint\limits_{t_2,\dotsc,t_k\ge 0}\Bigl(\int_0^{T/k} \Bigl(\prod_{i=1}^kg(kt_i)\Bigr)dt_1\Bigr)^2dt_2\dotsc dt_k\nonumber\\
&=\Bigl(\int_0^\infty g(kt_1)dt_1\Bigr)^2\Bigl(\int_0^\infty g(kt)^2dt\Bigr)^{k-1}=k^{-k-1}\gamma^{k-1}\Bigl(\int_{0}^\infty g(u)du\Bigr)^2,\label{eq:Jk'}\\
E_k&=\idotsint\limits_{\substack{t_2,\dotsc,t_k\ge 0\\ \sum_{i=2}^k t_i>1-T/k}}\Bigl(\int_{0}^{T/k} \Bigl(\prod_{i=1}^k g(kt_i)\Bigr) dt_1\Bigr)^2dt_2\dotsc dt_k\nonumber\\
&= k^{-k-1}\Bigl(\int_0^\infty g(u)du\Bigr)^2 \idotsint\limits_{\substack{u_2,\dotsc,u_k\ge 0\\ \sum_{i=2}^ku_i>k-T}}\Bigl(\prod_{i=2}^k g(u_i)^2\Bigr)du_2\dotsc du_k.\label{eq:EkBound}
\end{align}
First we wish to show the error integral $E_k$ is small. We do this by comparison with a second moment. We expect the bound \eqref{eq:EkBound} for $E_k$ to be small if the center of mass of $g^2$ is strictly less than $(k-T)/(k-1)$. Therefore we introduce the restriction on $g$ that
\begin{equation}
\mu=\frac{\int_0^\infty ug(u)^2du}{\int_0^\infty g(u)^2du}<1-\frac{T}{k}.\label{eq:MuBound}
\end{equation}
To simplify notation, we put $\eta=(k-T)/(k-1)-\mu>0$. If $\sum_{i=2}^ku_i>k-T$ then $\sum_{i=2}^ku_i>(k-1)(\mu+\eta)$, and so we have
\begin{equation}
1\le \eta^{-2}\Bigl(\frac{1}{k-1}\sum_{i=2}^ku_i-\mu\Bigr)^2.\label{eq:MomentBound}
\end{equation}
Since the right hand side of \eqref{eq:MomentBound} is non-negative for all $u_i$, we obtain an upper bound for $E_k$ if we multiply the integrand by $\eta^{-2}(\sum_{i=2}^ku_i/(k-1)-\mu)^2$, and then drop the requirement that $\sum_{i=1}^ku_i>k-T$. This gives us
\begin{align}
E_k&\le \eta^{-2}k^{-k-1}\Bigl(\int_0^\infty g(u)du\Bigr)^2\int_0^\infty\dotsi\int_0^\infty \Bigl(\frac{\sum_{i=2}^ku_i}{k-1}-\mu\Bigr)^2\Bigl(\prod_{i=2}^kg(u_i)^2\Bigr) du_2\dotsc du_k.
\end{align}
We expand out the inner square. All the terms which are not of the form $u_j^2$ we can calculate explicitly as an expression in $\mu$ and $\gamma$. We find
\begin{align}
\int_0^\infty\dotsi\int_0^\infty\Bigl(\frac{2\sum_{2\le i<j\le k}u_iu_j}{(k-1)^2}-\frac{2\mu\sum_{i=2}^ku_i}{k-1}+\mu^2\Bigr)\Bigl(\prod_{i=2}^kg(u_i)^2\Bigr) du_2\dotsc du_k&=\frac{-\mu^2\gamma^{k-1}}{k-1}.
\end{align}
For the $u_j^2$ terms we see that $u_j^2g(u_j)^2\le Tu_jg(u_j)^2$ from the support of $g$. Thus
\begin{align}
\int_0^\infty \dotsi\int_0^\infty u_j^2 \Bigl(\prod_{i=2}^k g(u_i)^2\Bigr)du_2\dotsc du_k& \le T \gamma^{k-2}\int_0^\infty u_j g(u_j)^2du_j= \mu T \gamma^{k-1}.
\end{align}
This gives
\begin{align}
E_k&\le \eta^{-2}k^{-k-1}\Bigl(\int_0^\infty g(u)du\Bigr)^2\Bigl(\frac{\mu T \gamma^{k-1}}{k-1}-\frac{\mu^2 \gamma^{k-1}}{k-1}\Bigr)\le \frac{\eta^{-2} \mu T k^{-k-1} \gamma^{k-1} }{k-1}\Bigl(\int_0^\infty g(u)du\Bigr)^2.\label{eq:EkBound}
\end{align}
Since $(k-1)\eta^2\ge k(1-T/k-\mu)^2$ and $\mu\le 1$, we find that putting together \eqref{eq:IkBound}, \eqref{eq:JkBound}, \eqref{eq:Jk'} and \eqref{eq:EkBound}, we obtain
\begin{equation}
\frac{kJ_k}{I_k}\ge \frac{(\int_0^\infty g(u)du)^2}{\int_0^\infty g(u)^2du}\Bigl(1-\frac{T}{k(1-T/k-\mu)^2}\Bigr).\label{eq:RatioExpression}
\end{equation}
To maximize our lower bound \eqref{eq:RatioExpression}, we wish to maximize $\int_0^T g(u)du$ subject to the constraints that $\int_0^T g(u)^2du=\gamma$ and $\int_0^T ug(u)^2du=\mu\gamma$. Thus we wish to maximize the expression
\begin{equation}
\int_0^T g(u)du-\alpha\Bigl(\int_0^T g(u)^2du-\gamma\Bigr)-\beta\Bigl(\int_0^T ug(u)^2du-\mu\gamma\Bigr)
\end{equation}
with respect to $\alpha,\beta$ and the function $g$. By the Euler-Lagrange equation, this occurs when $\frac{\partial}{\partial g}(g(t) - \alpha g(t)^2 -\beta t g(t)^2)=0$ for all $t\in[0,T]$. Thus we see that
\begin{equation}
g(t)=\frac{1}{2\alpha+2\beta t}\qquad \text{for $0\le t\le T$}.
\end{equation}
Since the ratio we wish to maximize is unaffected if we multiply $g$ by a positive constant, we restrict our attention to functions $g$ is of the form $1/(1+At)$ for $t\in[0,T]$ and for some constant $A>0$. With this choice of $g$ we find that
\begin{align}
\int_0^Tg(u)du&=\frac{\log(1+AT)}{A},\qquad \int_0^Tg(u)^2du=\frac{1}{A}\Bigl(1-\frac{1}{1+AT}\Bigr),\label{eq:Integrations}\\
\int_0^Tug(u)^2du&=\frac{1}{A^2}\Bigl(\log(1+AT)-1+\frac{1}{1+AT}\Bigr).
\end{align}
We choose $T$ such that $1+AT=e^A$ (which is close to optimal). With this choice we find that $\mu=1/(1-e^{-A})-A^{-1}$ and $T\le e^A/A$. Thus $1-T/k-\mu\ge A^{-1}(1-A/(e^A-1)-e^A/k)$. Substituting \eqref{eq:Integrations} into \eqref{eq:RatioExpression}, and then using these expressions, we find that
\begin{equation}
\frac{kJ_k}{I_k}\ge \frac{A}{1-e^{-A}}\Bigl(1-\frac{T}{k(1-T/k-\mu)^2}\Bigr)\ge A\Bigl(1-\frac{A e^A}{k(1-A/(e^A-1)-e^A/k)^2}\Bigr),
\end{equation}
provided the right hand side is positive. Finally, we choose $A=\log{k}-2\log\log{k}>0$. For $k$ sufficiently large we have
\begin{equation}
1-\frac{T}{k}-\mu\ge A^{-1}\Bigl(1-\frac{(\log{k})^3}{k}-\frac{1}{(\log{k})^2}\Bigr)>0,
\end{equation}
and so $\mu<1-T/k$, as required by our constraint \eqref{eq:MuBound}. This choice of $A$ gives
\begin{equation}
M_k\ge \frac{kJ_k}{I_k}\ge (\log{k}-2\log\log{k})\Bigl(1-\frac{\log{k}}{(\log{k})^2+O(1)}\Bigr)\ge \log{k}-2\log\log{k}-2
\end{equation}
when $k$ is sufficiently large.

\section{\texorpdfstring{Choice of weight for small $k$}{Choice of weight for small k}}
In this section we establish parts $(1)$ and $(2)$ of Proposition \ref{prpstn:FChoices}. In order to get a suitable lower bound for $M_k$ when $k$ is small, we will consider approximations to the optimal function $F$ of the form
\begin{equation}
F(t_1,\dotsc,t_k)=\begin{cases}P(t_1,\dotsc,t_k),\qquad &\text{if }(t_1,\dotsc,t_k)\in\mathcal{R}_k\\0,&\text{otherwise,}\end{cases}\label{eq:FChoice}
\end{equation}for polynomials $P$. By the symmetry of $\sum_{m=1}^kJ_k^{(m)}(F)$ and $I_k(F)$, we restrict our attention to polynomials which are symmetric functions of $t_1,\dotsc,t_k$. (If $F$ satisfies $\mathcal{L}_kF=\lambda F$ then $F_\sigma=F(\sigma(t_1),\dots,\sigma(t_k))$ also satisfies this for every permutation $\sigma$ of $t_1,\dots,t_k$. Thus the symmetric function which is the average of $F_\sigma$ over all such permutations would satisfy this eigenfunction equation, and so we expect there to be an optimal function which is symmetric.) Any such polynomial can be written as a polynomial expression in the power sum polynomials $P_j=\sum_{i=1}^kt_i^j$.
\begin{lmm}\label{lmm:Integration}
Let $P_j=\sum_{i=1}^kt_i^j$ denote the $j^{th}$ symmetric power sum polynomial. Then we have
\[\idotsint\limits_{\mathcal{R}_k}(1-P_1)^aP_j^bdt_1\dotsc dt_k=\frac{a!}{(k+jb+a)!}G_{b,j}(k),\]
where
\[G_{b,j}(x)=b!\sum_{r=1}^b\binom{x}{r}\sum_{\substack{b_1,\dotsc,b_r\ge 1\\ \sum_{i=1}^rb_i=b}}\prod_{i=1}^r\frac{(jb_i)!}{b_i!}\]
is a polynomial of degree $b$ which depends only on $b$ and $j$.
\end{lmm}
\begin{proof}
We first show by induction on $k$ that
\begin{equation}
\idotsint\limits_{\mathcal{R}_k}\Bigl(1-\sum_{i=1}^kt_i\Bigr)^a\prod_{i=1}^kt_i^{a_i}dt_1\dotsc dt_k=\frac{a!\prod_{i=1}^ka_i!}{(k+a+\sum_{i=1}^ka_i)!}.\label{eq:SimplifiedIntegration}
\end{equation}
We consider the integration with respect to $t_1$. The limits of integration are $0$ and $1-\sum_{i=2}^kt_i$ for $(t_2,\dotsc,t_k)\in\mathcal{R}_{k-1}$. By substituting $v=t_1/(1-\sum_{i=2}^kt_i)$ we find
\begin{align}
\int_0^{1-\sum_{i=2}^kt_i}\Bigl(1-\sum_{i=1}^kt_i\Bigr)^a\Bigl(\prod_{i=1}^kt_i^{a_i}\Bigr)dt_1 &=\Bigl(\prod_{i=2}^kt_i^{a_i}\Bigr)\Bigl(1-\sum_{i=2}^kt_i\Bigr)^{a+a_1+1}\int_0^1(1-v)^av^{a_1}dv\nonumber\\
&=\frac{a!a_1!}{(a+a_1+1)!}\Bigl(\prod_{i=2}^kt_i^{a_i}\Bigr)\Bigl(1-\sum_{i=2}^kt_i\Bigr)^{a+a_1+1}.
\end{align}
Here we used the beta function identity $\int_0^1t^a(1-t)^bdt=a!b!/(a+b+1)!$ in the last line. We now see \eqref{eq:SimplifiedIntegration} follows by induction.

By the binomial theorem,
\begin{equation}
P_j^b=\sum_{\substack{b_1,\dots,b_k\\ \sum_{i=1}^kb_i=b}}\frac{b!}{\prod_{i=1}^k b_i!}\prod_{i=1}^k t_i^{jb_i}.
\end{equation}
Thus, applying \eqref{eq:SimplifiedIntegration}, we obtain
\begin{equation}
\idotsint\limits_{\mathcal{R}_k}(1-P_1)^aP_j^bdt_1\dotsc dt_k=\frac{b!a!}{(k+a+jb)!}\sum_{\substack{b_1,\dotsc,b_k\\ \sum_{i=1}^k b_i=b}}\prod_{i=1}^k\frac{(jb_i)!}{b_i!}.
\end{equation}
For computations $b$ will be small, and so we find it convenient to split the summation depending on how many of the $b_i$ are non-zero. Given an integer $r$, there are $\binom{k}{r}$ ways of choosing $r$ of $b_1,\dotsc,b_k$ to be non-zero. Thus
\begin{equation}
\sum_{\substack{b_1,\dotsc,b_k\\ \sum_{i=1}^k b_i=b}}\prod_{i=1}^k\frac{(jb_i)!}{b_i!}=\sum_{r=1}^b\binom{k}{r}\sum_{\substack{b_1,\dots,b_r\ge 1\\ \sum_{i=1}^rb_i=b}}\prod_{i=1}^r\frac{(jb_i)!}{b_i!}.
\end{equation}
This gives the result.
\end{proof}
It is straightforward to extend Lemma \ref{lmm:Integration} to more general combinations of the symmetric power polynomials. In this paper we will concentrate on the case when $P$ is a polynomial expression in only $P_1$ and $P_2$ for simplicity. We comment the polynomials $G_{b,j}$ are not problematic to calculate numerically for small values of $b$. We now use Lemma \ref{lmm:Integration} to obtain a manageable expression for $I_k(F)$ and $J_k^{(m)}(F)$ with this choice of $P$.
\begin{lmm}\label{lmm:QuadraticForms}
Let $F$ be given in terms of a polynomial $P$ by \eqref{eq:FChoice}. Let $P$ be given in terms of a polynomial expression in the symmetric power polynomials $P_1=\sum_{i=1}^k t_i$ and $P_2=\sum_{i=1}^kt_i^2$ by $P=\sum_{i=1}^da_i(1-P_1)^{b_i}P_2^{c_i}$ for constants $a_i\in \mathbb{R}$ and non-negative integers $b_i,c_i$. Then for each $1\le m\le k$ we have
\begin{align*}
I_k(F)&=\sum_{1\le i,j\le d}a_ia_j\frac{(b_i+b_j)!G_{c_i+c_j,2}(k)}{(k+b_i+b_j+2c_i+2c_j)!},\\
J^{(m)}_k(F)&=\sum_{1\le i,j\le d}a_ia_j\sum_{c_1'=0}^{c_i}\sum_{c_2'=0}^{c_j}\binom{c_i}{c_1'}\binom{c_j}{c_2'}\frac{\gamma_{b_i,b_j,c_i,c_j,c_1',c_2'}G_{c_1'+c_2',2}(k-1)}{(k+b_i+b_j+2c_i+2c_j+1)!},
\end{align*}
where
\[
\gamma_{b_i,b_j,c_i,c_j,c_1',c_2'}=\frac{b_i!b_j!(2c_i-2c'_1)!(2c_j-2c'_2)!(b_i+b_j+2c_i+2c_j-2c_1'-2c_2'+2)!}{(b_i+2c_i-2c'_1+1)!(b_j+2c_j-2c'_2+1)!
},
\]
and where $G$ is the polynomial given by Lemma \ref{lmm:Integration}.
\end{lmm}
\begin{proof}
We first consider $I_k(F)$. We have, using Lemma \ref{lmm:Integration},
\begin{align}
I_k(F)&=\idotsint\limits_{\mathcal{R}_k}P^2dt_1\dotsc dt_k=\sum_{1\le i,j\le d}a_ia_j\idotsint\limits_{\mathcal{R}_k}(1-P_1)^{b_i+b_j}P_2^{c_i+c_j}dt_1\dotsc dt_k\nonumber\\
&=\sum_{1\le i,j\le d}a_ia_j\frac{(b_i+b_j)!G_{c_i+c_j,2}(k)}{(k+b_i+b_j+2c_i+2c_j)!}.
\end{align}
We now consider $J^{(m)}_k(F)$. Since $F$ is symmetric in $t_1,\dotsc,t_k$ we see that $J_k^{(m)}(F)$ is independent of $m$, and so it suffices to only consider $J^{(1)}_k(F)$. We have
\begin{align}
\int_0^{1-\sum_{i=2}^kt_i}(1-P_1)^bP_2^c dt_1&=\sum_{c'=0}^c \binom{c}{c'}\Bigl(\sum_{i=2}^kt_i^2\Bigr)^{c'}\int_0^{1-\sum_{i=2}^kt_i}\Bigl(1-\sum_{i=1}^kt_i\Bigr)^bt_1^{2c-2c'}dt_1\nonumber\\
&=\sum_{c'=0}^c \binom{c}{c'}(P'_2)^{c'}(1-P_1')^{b+2c-2c'+1}\int_0^1(1-u)^bu^{2c-2c'}du\nonumber\\
&=\sum_{c'=0}^c \binom{c}{c'}(P'_2)^{c'}(1-P_1')^{b+2c-2c'+1}\frac{b!(2c-2c')!}{(b+2c-2c'+1)!},
\end{align}
where $P_1'=\sum_{i=2}^kt_i$ and $P_2'=\sum_{i=2}^kt_i^2$. Thus
\begin{align}
\Bigl(\int_0^1Fdt_1\Bigr)^2&=\Bigl(\sum_{i=1}^da_i\int_0^{1-\sum_{j=2}^kt_j}(1-P_1)^{b_i}P_2^{c_i}dt_1\Bigr)^2\nonumber\\
&=\sum_{1\le i,j\le d}a_ia_j\sum_{c'_1=0}^{c_i}\sum_{c'_2=0}^{c_j}\binom{c_i}{c_1'}\binom{c_j}{c_2'}(P_2')^{c_1'+c_2'}(1-P_1')^{b_i+b_j+2c_i+2c_j-2c_1'-2c_2'+2}\nonumber\\
&\qquad\times \frac{b_i!b_j!(2c_i-2c'_1)!(2c_j-2c'_2)!}{(b_i+2c_i-2c'_1+1)!(b_j+2c_j-2c'_2+1)!
}.\label{eq:PrimeDecomposition1}
\end{align}
Applying Lemma \ref{lmm:Integration} again, we see that
\begin{equation}
\idotsint\limits_{\mathcal{R}_{k-1}}(1-P_1')^{b}(P_2')^{c'}dt_2\dotsc dt_k=\frac{b!}{(k+b+c-1)!}G_{c,2}(k-1).\label{eq:PrimeIntegration}
\end{equation}
Combining \eqref{eq:PrimeDecomposition1} and \eqref{eq:PrimeIntegration} gives the result.
\end{proof}
We see from Lemma \ref{lmm:QuadraticForms} that $I_k(F)$ and $\sum_{m=1}^kJ_k^{(m)}(F)$ can both be expressed as quadratic forms in the coefficients $\mathbf{a}=(a_1,\dots,a_d)$ of $P$. Moreover, these will be positive definite real quadratic forms. Thus in particular we find that
\begin{equation}
\frac{\sum_{m=1}^k J_k^{(m)}(F)}{I_k(F)}=\frac{\mathbf{a}^T A_2 \mathbf{a}}{\mathbf{a}^T A_1 \mathbf{a}},
\end{equation}
for two rational symmetric positive definite matrices $A_1,A_2$, which can be calculated explicitly in terms of $k$ for any choice of the exponents $b_i,c_i$. Maximizing expressions of this form has a known solution.
\begin{lmm}
Let $A_1,A_2$ be real, symmetric positive definite matrices. Then
\[\frac{\mathbf{a}^T A_2 \mathbf{a}}{\mathbf{a}^T A_1 \mathbf{a}}\]
is maximized when $\mathbf{a}$ is an eigenvector of $A_1^{-1}A_2$ corresponding to the largest eigenvalue of $A_1^{-1}A_2$. The value of the ratio at its maximum is this largest eigenvalue.
\end{lmm}
\begin{proof}
We see that multiplying $\mathbf{a}$ by a non-zero scalar doesn't change the ratio, so we may assume without loss of generality that $\mathbf{a}^TA_1\mathbf{a}=1$. By the theory of Lagrangian multipliers, $\mathbf{a}^TA_2\mathbf{a}$ is maximized subject to $\mathbf{a}^TA_1\mathbf{a}=1$ when
\begin{equation}
L(\mathbf{a},\lambda)=\mathbf{a}^T A_2\mathbf{a}-\lambda(\mathbf{a}^TA_1\mathbf{a}-1)
\end{equation}
is stationary. This occurs when (using the symmetricity of $A_1,A_2$)
\begin{equation}
0=\frac{\partial L}{\partial a_i}=((2A_2-2\lambda A_1)\mathbf{a})_i,
\end{equation}
for each $i$. This implies that (recalling that $A_1$ is positive definite so invertible)
\begin{equation}
A_1^{-1}A_2\mathbf{a}=\lambda\mathbf{a}.
\end{equation}
It then is clear that $\mathbf{a}^TA_1\mathbf{a}=\lambda^{-1}\mathbf{a}^T A_2\mathbf{a}$.
\end{proof}
\begin{proof}[Proof of parts $(1)$ and $(2)$ of Proposition \ref{prpstn:FChoices}]
To establish Proposition \ref{prpstn:FChoices} we rely on some computer calculation to calculate a lower bound for $M_k$. We let $F$ be given in terms of a polynomial $P$ by \eqref{eq:FChoice}. We let $P$ be given by a polynomial expression in $P_1=\sum_{i=1}^kt_i$ and $P_2=\sum_{i=1}^kt_i^2$ which is a linear combination of all monomials $(1-P_1)^bP_2^c$ with $b+2c\le 11$. There are \num{42} such monomials, and with $k=105$ we can calculate the $42\times 42$ rational symmetric matrices $A_1$ and $A_2$ corresponding to the coefficients of the quadratic forms $I_k(F)$ and $\sum_{m=1}^k J_k^{(m)}(F)$. We then find\footnote{An ancillary  \textit{Mathematica}\textsuperscript{\textregistered} file detailing these computations is available alongside this paper at \url{www.arxiv.org}.} that the largest eigenvalue of $A_1^{-1}A_2$ is
\begin{equation}
\lambda\approx 4.0020697\dotsc >4.
\end{equation}
Thus $M_{105}> 4$. This verifies part $(2)$ of Proposition \ref{prpstn:FChoices}. We comment that by taking a rational approximation to the corresponding eigenvector, we can verify this lower bound by calculating the ratio $\sum_{m=1}^kJ_k^{(m)}(F)/I_k(F)$ using only exact arithmetic.

For part $(1)$ of Proposition \ref{prpstn:FChoices}, we take $k=5$ and 
\begin{equation}P=(1-P_1)P_2+\frac{7}{10}(1-P_1)^2+\frac{1}{14}P_2-\frac{3}{14}(1-P_1).\end{equation}
With this choice we find that
\begin{equation}
M_5\ge \frac{\sum_{m=1}^kJ_k^{(m)}(F)}{I_k(F)}=\frac{\num{1417255}}{\num{708216}}>2.
\end{equation}
This completes the proof of Proposition \ref{prpstn:FChoices}.
\end{proof}
\section{Acknowledgements}
The author would like to thank Andrew Granville, Roger Heath-Brown, Dimitris Koukoulopoulos and Terence Tao for many useful conversations and suggestions. 

The work leading to this paper was started whilst the author was a D.Phil student at Oxford and funded by the EPSRC (Doctoral Training Grant EP/P505216/1), and was finished when the author was a CRM-ISM postdoctoral fellow at the Universit\'e de Montr\'eal.
\bibliographystyle{plain}
\bibliography{SmallGaps}
\end{document}